\documentclass[a4paper,reqno,10pt]{amsart}

\usepackage{hyperref}
\usepackage{a4wide}

\usepackage{amssymb}
\usepackage{amstext}
\usepackage{amsmath}
\usepackage{amscd}
\usepackage{amsthm}
\usepackage{amsfonts}

\usepackage{graphicx}
\usepackage{latexsym}
\usepackage{mathrsfs}

\usepackage{enumitem}
\setlist[enumerate,1]{label={\upshape(\arabic*)}}
\setlist[enumerate,2]{label={\upshape(\alph*)}}

\usepackage{tikz}
\usetikzlibrary{cd,arrows,matrix,backgrounds,positioning,calc,decorations.markings,decorations.pathmorphing,decorations.pathreplacing}
\tikzset{black/.style={circle,fill=black,inner sep=3pt,outer sep=3pt},
         white/.style={circle,fill=white,draw=black,inner sep=3pt,outer sep=3pt},
}

\usepackage{array}
\newcolumntype{C}{>{$}c<{$}}

\usepackage{makecell}

\newtheorem{theorem}{Theorem}[section]
\newtheorem{theoremi}{Theorem}

\newtheorem{propositioni}[theoremi]{Proposition}

\newtheorem{corollary}[theorem]{Corollary}
\newtheorem{lemma}[theorem]{Lemma}
\newtheorem*{lemma*}{Lemma}
\newtheorem*{theorem*}{Theorem}
\newtheorem{proposition}[theorem]{Proposition}
\newtheorem{definition-proposition}[theorem]{Definition-Proposition}

\newtheorem{conjecture}[theorem]{Conjecture}

\theoremstyle{definition}
\newtheorem{definition}[theorem]{Definition}

\newtheorem{remark}[theorem]{Remark}
\newtheorem{example}[theorem]{Example}

\newcommand{\qedb}{\hfill\blacksquare}

\newcommand{\la}{\langle}
\newcommand{\ra}{\rangle}
\newcommand{\al}{\alpha}
\newcommand{\be}{\beta}
\newcommand{\ga}{\gamma}
\newcommand{\vare}{\varepsilon}

\newcommand{\CC}{\mathcal{C}}

\newcommand{\GG}{\mathcal{G}}

\newcommand{\DD}{\mathcal{D}}

\newcommand{\KKK}{\mathsf{K}}
\newcommand{\FF}{\mathcal{F}}
\newcommand{\FFF}{\mathsf{F}}

\newcommand{\RR}{\mathcal{R}}
\renewcommand{\SS}{\mathcal{S}}

\newcommand{\Z}{\mathbb{Z}}

\newcommand{\R}{\mathbb{R}}

\newcommand{\Ext}{\operatorname{Ext}\nolimits}

\newcommand{\Hom}{\operatorname{Hom}\nolimits}

\newcommand{\End}{\operatorname{End}\nolimits}

\newcommand{\RHom}{\mathbf{R}\strut\kern-.2em\operatorname{Hom}\nolimits}

\newcommand{\Image}{\operatorname{Im}\nolimits}
\newcommand{\Kernel}{\operatorname{Ker}\nolimits}

\newcommand{\GL}{\operatorname{GL}\nolimits}
\newcommand{\SL}{\operatorname{SL}\nolimits}

\newcommand{\im}{\Image}
\renewcommand{\ker}{\Kernel}
\newcommand{\un}{\underline}
\newcommand{\ov}{\overline}

\newcommand{\cov}{\gtrdot}
\newcommand{\covd}{\lessdot}

\newcommand{\ot}{\leftarrow}

\DeclareMathOperator{\moduleCategory}{\mathsf{mod}} \renewcommand{\mod}{\moduleCategory}

\DeclareMathOperator{\ind}{\mathsf{ind}}
\DeclareMathOperator{\simp}{\mathsf{sim}}

\DeclareMathOperator{\torf}{\mathsf{torf}}
\DeclareMathOperator{\Hasse}{\mathsf{Hasse}}
\DeclareMathOperator{\brick}{\mathsf{brick}}

\DeclareMathOperator{\Sub}{\mathsf{Sub}}

\DeclareMathOperator{\Filt}{\mathsf{Filt}}

\DeclareMathOperator{\add}{\mathsf{add}}

\DeclareMathOperator{\inv}{\mathsf{inv}}
\DeclareMathOperator{\Binv}{\mathsf{Binv}}

\DeclareMathOperator{\Inv}{\mathsf{Inv}}
\DeclareMathOperator{\BInv}{\mathsf{BInv}}

\DeclareMathOperator{\supp}{\mathsf{supp}}

\DeclareMathOperator{\udim}{\un{\dim}}

\newcommand{\iso}{\cong}

\newcommand{\defl}{\twoheadrightarrow}

\newcommand{\EE}{\mathcal{E}}

\numberwithin{equation}{section}

\begin{document}
\title[Bruhat inversions and torsion-free classes over preprojective algebras]{Bruhat inversions in Weyl groups\\ and torsion-free classes over preprojective algebras}

\author[H. Enomoto]{Haruhisa Enomoto}
\address{Graduate School of Mathematics, Nagoya University, Chikusa-ku, Nagoya. 464-8602, Japan}
\email{m16009t@math.nagoya-u.ac.jp}
\begin{abstract}
  For an element $w$ of the simply-laced Weyl group, Buan-Iyama-Reiten-Scott defined a subcategory $\mathcal{F}(w)$ of a module category over a preprojective algebra of Dynkin type. This paper aims at studying categorical properties of $\mathcal{F}(w)$ via its connection with the root system.
  We show that by taking dimension vectors, simple objects in $\mathcal{F}(w)$ bijectively correspond to Bruhat inversion roots of $w$.
  As an application, we obtain a combinatorial criterion for $\mathcal{F}(w)$ to satisfy the Jordan-H\"older property (JHP).
  To achieve this, we develop a method to find simple objects in a general torsion-free class by using a brick sequence associated to a maximal green sequence of it.
  For type A case, we give a diagrammatic construction of simple objects, and show that (JHP) can be characterized via a forest-like permutation, introduced by Bousquet-M\'elou and Butler in the study of Schubert varieties.
\end{abstract}

\maketitle

\tableofcontents

\section{Introduction}
This paper focuses on the interplay between the preprojective algebras of Dynkin type and the root system. More precisely, we study a certain subcategory $\FF(w)$ of the module category of the preprojective algebra via an inversion set in the root system.

\subsection{Background}
Let $\Phi$ be the simply-laced root system of type $X$ and $W$ the corresponding Weyl group. Let $Q$ be a quiver of type $X$, that is, the underlying graph of $Q$ is the Dynkin diagram $X$. Then the celebrated Gabriel's theorem gives a bijection between indecomposable representations of $Q$ and positive roots in $\Phi$, by taking dimension vectors.

A \emph{preprojective algebra $\Pi$} of $\Phi$ is a finite-dimensional algebra which unifies the representation theory of all quivers of type $X$, and has a lot of symmetry compared to path algebras.
This algebra has been one of the most important objects in the representation theory of algebras, for example, \cite{AIRT,BIRS,GLS,IRRT,mizuno}, and also plays an important role in the theory of crystal bases of quantum groups, for example, \cite{Lus,KS}.

In this paper, we focus on a certain subcategory $\FF(w)$ of $\mod\Pi$ associated to an element $w$ of $W$ introduced by Buan-Iyama-Reiten-Scott \cite{BIRS} (under the name $\CC_w$). This category has a nice structure related to cluster algebras, that is, a stably 2-Calabi-Yau Frobenius category admitting a cluster-tilting object. Indeed, Geiss-Leclerc-Schr\"oer later \cite{GLS} proved that $\FF(w)$ gives a categorification of the cluster algebra structure on  the coordinate ring of the unipotent cell in the complex simple Lie group of Dynkin type $X$.

The category $\FF(w)$ naturally arises also from the viewpoint of the representation theory of algebras, as well as the lattice theoretical study of the Weyl group.
This category is a \emph{torsion-free class} in $\mod \Pi$, that is, closed under submodules and extensions. Mizuno \cite{mizuno} proved that the map $w \mapsto \FF(w)$ is actually a bijection from $W$ to the set $\torf\Pi$ of all torsion-free classes in $\mod\Pi$.
He also proved that this bijection is an isomorphism of lattices, where we endow $W$ with the right weak order and $\torf\Pi$ the inclusion order. Via this isomorphism, lattice theoretical properties of $W$ were investigated in \cite{IRRT,DIRRT}.
\subsection{Main results}
It is natural to expect that categorical properties of $\FF(w)$ are related to combinatorial properties of $w$. In this direction, we prove the two main results: we classify \emph{simple objects} in $\FF(w)$ (Theorem A), and  give a criterion for the validity of the \emph{Jordan-H\"older type theorem} in $\FF(w)$ (Theorem C).

A $\Pi$-module $M$ in $\FF(w)$ is called a \emph{simple object} in $\FF(w)$ if there is no non-trivial submodule $L$ of $M$ satisfying $L,M/L \in \FF(w)$. This notion was introduced in the context of Quillen's exact categories, and has been investigated by several papers such as \cite{eno,bhlr}.
In \cite{eno}, the author classified simple objects in a torsion-free class in $\mod kQ$ for type A case, and the original motivation of this paper is to generalize this to other Dynkin types and to preprojective algebras.

Our strategy is to consider $\FF(w)$ via the root system $\Phi$. For $M \in \mod\Pi$, we can regard its dimension vector as a vector in the ambient space of $\Phi$ naturally. Let $\inv(w)$ be the set of \emph{inversions} of $w$, positive roots which are sent to negative by $w^{-1}$. Then any $M \in \FF(w)$, its dimension vector $\udim M$ is a non-negative integer linear combination of inversions of $w$ (Corollary \ref{binv:cor:diminv}).
In parallel with simples in $\FF(w)$, it is natural to consider an inversion of $w$ which cannot be written as a sum of inversions of $w$ non-trivially. We call such a root a \emph{Bruhat inversion} (Definition \ref{binv:def:binv}, Theorem \ref{binv:thm:binvchar}). Then the first main result of this paper is the following:
\begin{theoremi}[= Corollary \ref{binv:cor:main}]\label{binv:thmA}
  Let $\Pi$ be the preprojective algebra of Dynkin type and $w$ an element of the corresponding Weyl group $W$. Then by taking dimension vectors, we have the bijection between the following two sets:
  \begin{enumerate}
    \item The set of isomorphism classes of simple objects in $\FF(w)$.
    \item The set of Bruhat inversions of $w$.
  \end{enumerate}
\end{theoremi}
This immediately deduces the similar result for the case of the path algebra $kQ$. A torsion-free class in $\mod kQ$ bijectively corresponds to a \emph{$c_Q$-sortable element} of $W$ by \cite{IT}. For such an element $w$ of $W$, we have a torsion-free class $\FF_Q(w)$ in $\mod kQ$, which is actually equal to the restriction of $\FF(w)$ to $\mod kQ$ (Proposition \ref{binv:prop:fwadd}).
Then we have the same result for $\FF_Q(w)$:
\begin{theoremi}[= Theorem \ref{binv:thm:pathmain}]\label{binv:thmB}
  Let $Q$ be the Dynkin quiver, $w$ a $c_Q$-sortable element of the Weyl group $W$ and $\FF_Q(w)$ the corresponding torsion-free class in $\mod kQ$. Then by taking dimension vectors, we have the bijection between the following two sets:
  \begin{enumerate}
    \item The set of isomorphism classes of simple objects in $\FF_Q(w)$.
    \item The set of Bruhat inversions of $w$.
  \end{enumerate}
\end{theoremi}

As an application, we can characterize the validity of the \emph{Jordan-H\"older property (JHP)} in $\FF(w)$ or $\FF_Q(w)$. We say that a torsion-free class $\FF$ \emph{satisfies (JHP)} if for any $M$ in $\FF$, any relative composition series of $M$ in $\FF$ are equivalent (see Definition \ref{binv:def:jhp} for the precise definition). By using the characterization of (JHP) obtained in \cite{eno}, we prove the following second main result.
\begin{theoremi}[{= Theorem \ref{binv:thm:jhpmain}, Proposition \ref{binv:prop:jhpforest}}]\label{binv:thmC}
  Let $\Pi$ be the preprojective algebra of Dynkin type and $w$ an element of the corresponding Weyl group $W$. Then the following are equivalent:
  \begin{enumerate}
    \item $\FF(w)$ satisfies (JHP).
    \item Bruhat inversions of $w$ are linearly independent.
    \item The number of Bruhat inversions of $w$ is equal to that of supports of $w$.
  \end{enumerate}
  Moreover, for the type A case, the above statements are equivalent to the following:
  \begin{enumerate}[resume]
    \item $w$ is forest-like in the sense of \cite{bmb}, that is, its Bruhat inversion graph is acyclic.
  \end{enumerate}
\end{theoremi}
Here a \emph{support} of $w$ is a vertex $i$ in the Dynkin diagram $X$ such that the reduced expression of $w$ contains the simple reflection $s_i$. See Appendix \ref{binv:sec:A2} for the details on forest-like permutations. The same result also holds for the case of path algebras (Corollary \ref{binv:cor:jhppath}).

To show these results, we develop a method to find simple objects in a given torsion-free class $\FF$ by using a \emph{brick sequence} associated to a maximal green sequence of $\FF$. A \emph{maximal green sequence} of $\FF$ is just a saturated chain $0 = \FF_0 \covd \FF_1 \covd \cdots \covd \FF_l = \FF$ of torsion-free classes. We can associate to it a sequence of bricks (modules $B$ such that $\End_\Lambda(B)$ is a division ring) by using the brick labeling introduced in \cite{DIRRT}. Once we can compute one brick sequence of $\FF$, the following gives a way to determine all simple objects in $\FF$:
\begin{propositioni}[= Corollary \ref{binv:cor:simplabel}]\label{binv:propD}
  Let $\Lambda$ be a finite-dimensional algebra and $\FF$ a torsion-free class in $\mod\Lambda$. Suppose that there is a maximal green sequence of $\FF$, and let $B_1,\dots,B_l$ be the associated brick sequence. Then the following hold:
  \begin{enumerate}
    \item Every simple object in $\FF$ is isomorphic to $B_i$ for some $i$.
    \item For $1 \leq i \leq l$, the following are equivalent:
    \begin{enumerate}
      \item $B_i$ is a simple object in $\FF$.
      \item Every morphism $B_i \to B_j$ with $i < j$ is either zero or injective.
      \item There is no surjection $B_i \defl B_j$ with $i < j$.
    \end{enumerate}
  \end{enumerate}
\end{propositioni}

\subsection{Organization}
This paper is organized as follows. In Section \ref{binv:sec:2}, we give root-theoretical preliminaries and results. More precisely, we characterize Bruhat inversions (Theorem \ref{binv:thm:binvchar}). Next we introduce a root sequence associated to a reduced expression of an element of $W$, which plays an important role later.
In Section \ref{binv:sec:3}, we develop a general theory of simple objects in a torsion-free class, and prove Proposition \ref{binv:propD}. In Section \ref{binv:sec:4}, we focus on torsion-free classes over preprojective algebras of Dynkin type. We show that a brick sequence of $\FF(w)$ categorifies a root sequence of $w$, and prove Theorem \ref{binv:thmA}, \ref{binv:thmB}.
In Section \ref{binv:sec:5}, we deduce several results on path algebras using preprojective algebras, and prove Theorem \ref{binv:thmC}.
In the appendix, we give two combinatorial interpretation of the results for type A case. In Section \ref{binv:sec:A1}, we give a diagrammatic construction of simple objects in $\FF(w)$ using arc diagrams, and in Section \ref{binv:sec:A2}, we show that $\FF(w)$ satisfies (JHP) if and only if $w$ is a forest-like permutation.

\subsection{Conventions and notation}
Throughout this paper, \emph{we assume that all categories are skeletally small}, that is, the isomorphism classes of objects form a set. In addition, \emph{all subcategories are assumed to be full and closed under isomorphisms}.
For a Krull-Schmidt category $\EE$, we denote by $\ind\EE$ the set of isomorphism classes of indecomposable objects in $\EE$. We denote by $|X|$ the number of non-isomorphic indecomposable direct summands of $X$.

For a poset $P$ and two elements $a,b \in P$ with $a \leq b$, we denote by $[a,b]$ the \emph{interval poset} $[a,b]:=\{ x \in P \, | \, a \leq x \leq b \}$ with the obvious partial order.
For $x,y \in P$, we say that $x$ \emph{covers} $y$ if $x > y$ holds and there exists no $z \in P $ with $x > z > y$. In this case, we write $x \cov y$ or $y \covd x$.

For a set $A$, we denote by $\# A$ the cardinality of $A$.

\medskip\noindent
{\bf Acknowledgement.}
The author would like to thank his supervisor Osamu Iyama for many helpful comments and discussions. He is also grateful to Yuya Mizuno for explaining to him arc diagrams and bricks for type A case. This work is supported by JSPS KAKENHI Grant Number JP18J21556.

\section{Preliminaries on root system}\label{binv:sec:2}
In this section, we give some results on root systems which we need later. In particular, we give a definition and characterization of Bruhat inversions, and introduce the notion of a root sequence.
\subsection{Basic definitions}
First we recall some basic definitions and properties of root systems. We refer the reader to \cite{hum1,hum2} for the details.

Let $V$ be the Euclidean space, that is, a finite-dimensional $\R$-vector space with the positive definite symmetric bilinear form $(-,-) \colon V \times V \to \R$.

For two vectors $\al, \be \in V$ with $\al \neq 0$, we put
\[
\la \be, \al \ra := 2 \frac{(\be,\al)}{(\al, \al)}.
\]
Note that $\la-,\al\ra$ is linear but $\la \al,-\ra$ is not. For $\al \in V$ with $\al \neq 0$, we denote by $t_\al \colon V \to V$ the reflection with respect to $\al$, that is,
\[
t_\al(\be) = \be - \la \be, \al \ra \al.
\]

\begin{definition}
  A subset $\Phi$ of the Euclidean space $V$ is called a \emph{root system} if it satisfies the following axioms:
  \begin{enumerate}
    \item[(R0)] $\Phi$ is a finite subset of $V$ which spans $V$ as an $\R$-vector space, and does not contain $0$.
    \item[(R1)] $\Phi \cap \R \al = \{ \al, -\al \}$ for every $\al \in \Phi$.
    \item[(R2)] $t_\al(\Phi) = \Phi$ for every $\al \in \Phi$.
    \item[(R3)] $\la \be,\al\ra \in \Z$ for every $\al,\be \in \Phi$.
  \end{enumerate}
  A root system is called \emph{simply-laced} if it satisfies the following condition.
  \begin{enumerate}
    \item[(R4)] $\la \al,\be \ra = \la \be,\al \ra$ for every $\al,\be \in \Phi$.
  \end{enumerate}
\end{definition}
Obviously (R4) is equivalent to that for every $\al,\be \in \Phi$, if $(\al,\be) \neq 0$, then $\al$ and $\be$ have the same length, that is, $(\al,\al) = (\be,\be)$.

Possible values of integers $\la \al,\be \ra$ in (R3) are very limited as follows.
\begin{proposition}[{\cite[9.4]{hum1}}]\label{binv:prop:cartanint}
  Let $\Phi$ be a root system, and let $\al$ and $\be$ be two roots in $\Phi$ with $(\al,\al) \leq (\be,\be)$. Then the following hold:
  \begin{enumerate}
    \item $\la \al,\be \ra \cdot \la \be,\al \ra \geq 0$ and
    $(|\la\al,\be\ra|, |\la \be,\al\ra|) \in \{ (0,0), (1,1),  (1,2),  (1,3), (2,2) \}$.
    \item $\be = \pm \al$ if and only if $(|\la\al,\be\ra|, |\la \be,\al\ra|) = (2,2)$.
    \item $\Phi$ is simply-laced if and only if $\la \al,\be \ra \in \{-1,0,1\}$ for every $\al,\be \in \Phi$ with $\al \neq \be$.
  \end{enumerate}
\end{proposition}
\emph{Throughout this section, we will use the following notation:}
\begin{itemize}
  \item $\Phi$ is a root system in $V$.
  \item We fix a choice of simple roots $\Delta$ of $\Phi$.
  \item $\Phi^+$ (resp. $\Phi^-$) is the set of positive roots (resp. negative roots) in $\Phi$ with respect to $\Delta$.
  \item $W$ is the Weyl group associated with $\Phi$, that is, $W$ is a subgroup of $\GL(V)$ generated by $t_\al$ with $\al \in \Phi$.
  \item $T\subset W$ is a set of reflections in $W$, that is, $T= \{ t_\al \in W \, | \, \al \in \Phi \}$.
  \item We often write $s_\al = t_\al$ if $\al$ is a simple root.
\end{itemize}
It is well-known that $W$ is generated by simple reflections. For $w\in W$, let
\[
w = s_1 \cdots s_l
\]
be such an expression. Then this expression is called a \emph{reduced expression} if $l$ is minimal among all such expressions. In this case, $l$ is called the \emph{length} of $w$ and we write $\ell(w) := l$.

On the reduced expression, we will need the following \emph{exchange property} later.
\begin{lemma}[{\cite[1.6, 1.7]{hum2}}]\label{binv:lem:exchange}
  Let $w = s_1 \cdots s_l$ be a reduced expression of $w \in W$, and let $\al$ be a simple roots. Then the following are equivalent:
  \begin{enumerate}
    \item $s_1 \cdots s_l s_{\al}$ is not reduced, that is, $\ell(w s_\al) < \ell(w)$.
    \item There exists $i$ such that $s_1 \cdots s_l s_\al = s_1 \cdots \widehat{s_i} \cdots s_l$ ($s_i$ omitted).
    \item $w(\al) \in \Phi^{-}$.
  \end{enumerate}
\end{lemma}

\subsection{Inversion sets and root sequences}
For $w \in W$, we define $\inv (w)$ by
\[
\inv (w) := \Phi^+ \cap w(\Phi^-),
\]
that is, $\inv(w)$ is the set of positive roots which are sent to negative roots by $w^{-1}$. We call an element of $\inv(w)$ an \emph{inversion of $w$}.
This set plays an important role in this paper, since it corresponds to the set of dimension vectors of bricks in a torsion-free class $\FF(w)$ over the preprojective algebra (Corollary \ref{binv:cor:diminv}).

The following description of inversion sets is well-known.
\begin{proposition}[{\cite[1.7]{hum2}}]\label{binv:prop:invdist}
  Let $w$ be an element of $W$. Take a reduced expression $w= s_{\al_1} \cdots s_{\al_l}$ of $w$ with $\al_i \in \Delta$. Then we have
    \[
    \inv(w) = \{ \al_1,\;  s_{\al_1}(\al_2),\; s_{\al_1} s_{\al_2} (\al_3),\; \dots,\; s_{\al_1}\cdots s_{\al_{l-1}}(\al_l) \},
    \]
    and all the elements above are distinct. In particular, we have $\#\inv(w) = \ell(w)$.
\end{proposition}
By using this, we can easily show the following property.
\begin{lemma}\label{binv:lem:invcon}
  Let $v$ and $w$ elements of $W$ satisfying $\ell(vw) = \ell(v) + \ell(w)$. Then we have $\inv(vw) = \inv(v)\sqcup v (\inv(w))$.
\end{lemma}

If we choose a reduced expression of $w$, then Proposition \ref{binv:prop:invdist} gives a sequence of positive roots. It turns out that the order of appearance are important for our purpose. This leads to the notion of \emph{root sequences}.
\begin{definition}\label{binv:def:rootseq}
  Let $W$ be the Weyl group of a root system $\Phi$.
  \begin{enumerate}
    \item Let $w$ be an element of $W$ and $w = s_{\al_1}\cdots s_{\al_l}$ a reduced expression of $w$. Then a \emph{root sequence of $w$ associated to this expression} is a  (ordered) sequence of positive roots
    \[
    \al_1,\; s_{\al_1}(\al_2),\; s_{\al_1} s_{\al_2} (\al_3),\; \dots,\; s_{\al_1}\cdots s_{\al_{l-1}}(\al_l).
    \]
    \item A \emph{root sequence} is a sequence of roots which arises as a root sequence of some reduced expression of some element $w$ in $W$. We call such a sequence a \emph{root sequence of $w$}.
  \end{enumerate}
\end{definition}

The notion of root sequences appeared in several papers: they are called \emph{compatible (convex) orderings} in \cite{papi}, \emph{reflection orderings} in \cite{dyer}. We borrowed the terminology \emph{root sequences} from \cite{gl,fs}.

We will use the following characterization of inversion sets and root sequences due to Papi later.
\begin{theorem}[\cite{papi}]\label{binv:thm:rootseqchar}
  A sequence $\RR$ of positive roots is a root sequence if and only if the following two conditions are satisfied for any pair of positive roots $\al, \be$ satisfying $\al + \be \in \Phi^+$:
  \begin{enumerate}
    \item If $\al$ and $\be$ appear in $\RR$, then $\al + \be$ appears between $\al$ and $\be$ in $\RR$.
    \item If $\al + \be$ appears in $\RR$, then one of $\al$ and $\be$ appears and precedes $\al + \be$ in $\RR$.
  \end{enumerate}
\end{theorem}
To give examples of root sequences (and its connection with \emph{brick sequence} defined in the next section), it is convenient to introduce the \emph{right weak order $\leq_R$} on $W$ and its \emph{Hasse quiver}. Define a quiver $\Hasse(W,\leq_R)$ as follows:
\begin{itemize}
  \item A vertex set of $\Hasse (W,\leq_R)$ is $W$.
  \item We draw an arrow $v \ot vs$ if $\ell(vs) = \ell(v)+1$ for $v \in W$ and a simple reflection $s \in W$.
\end{itemize}
Then define a partial order $\leq_R$ on $W$ by $v \leq_R w$ if and only if there is a path from $w$ to $v$ in $\Hasse(W,\leq_R)$. It is known that $(W,\leq_R)$ is actually a lattice, see e.g. \cite[3.2]{bb}, and $\Hasse(W,\leq_R)$ is actually a Hasse quiver of $(W,\leq_R)$ defined later.

By construction, each reduced expression of $w = s_1 s_2 \cdots s_l$ gives a reverse path $e \ot s_1 \ot s_1 s_2 \ot \cdots \ot w$ from $e$ to $w$ in $\Hasse(W,\leq_R)$, and this correspondence is a bijection.
We define the \emph{root labeling} of arrows in $\Hasse(W,\leq_R)$ by attaching a positive root $v(\al)$ to an arrow $v \ot vs_\al$. Then a root sequence associated to a reduced expression $w=s_1 \cdots s_l$ is obtained by reading labels of the corresponding reverse path $e \ot s_1 \ot \cdots \ot w$.
\begin{example}\label{binv:ex:ainv}
  Let $\Phi$ be the root system of type $A_3$ with its Dynkin graph $1$---$2$---$3$. We denote by $s_i$ (resp. $\al_i$) the simple reflection (resp. simple root) associated to the vertex $i$ for $1 \leq i \leq 3$. Consider $w=s_1 s_2 s_3 s_1 s_2 \in W$. Figure \ref{binv:fig:ainv} shows all the reverse paths from $e$ to $w$ and its root labeling.
  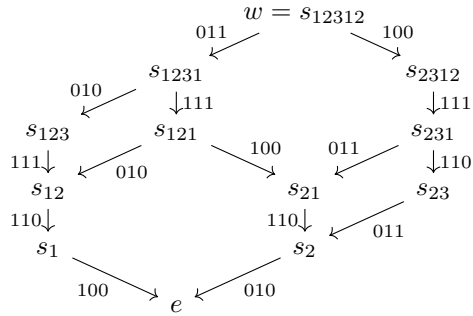
\begin{figure}[h]
  \begin{tikzcd}[column sep=tiny, row sep=small]
    &[.5cm] & w= s_{12312} \ar[dl, "011"'] \ar[dr, "100"] \\
    & s_{1231} \dar["111"] \ar[dl,"010"'] & & s_{2312} \dar["111"] \\
    s_{123}\dar["111"'] & s_{121}\ar[dl, "010"] \ar[dr, "100"] & & s_{231} \dar["110"] \ar[dl,"011"']\\
    s_{12}\dar["110"'] & & s_{21}\dar["110"'] & s_{23} \ar[dl, "011"]\\
    s_1\ar[rd, "100"'] & & s_2 \ar[ld, "010"]&  \\
    & e
  \end{tikzcd}
  \caption{Root sequences of $w=s_{12312}$}
  \label{binv:fig:ainv}
  \end{figure}
  Here we write $s_{1231}:=s_1 s_2 s_3 s_1$ and $110:=\al_1 + \al_2$ for example.
  Then the left most path corresponds to a reduced expression $s_{12312}$, and the right most corresponds to $s_{23121}$.
  Their associated root sequences are $100,110,111,010,011$ and $010,011,110,111,100$ respectively.
\end{example}

\begin{example}\label{binv:ex:dinv}
  Let $\Phi$ be the root system of type $D_4$, whose Dynkin diagram is as follows:
  \[
  \begin{tikzcd}[column sep=tiny, row sep=tiny]
    & 2 \dar[dash]\\
    1 \rar[dash] & 0 \rar[dash] & 3
  \end{tikzcd}
  \]
  Consider $w=s_{012301230}$ (here we use the same notation as in Example \ref{binv:ex:ainv}). Figure \ref{binv:fig:dinv} shows all the reduced expressions of $w$ and its root sequence.
  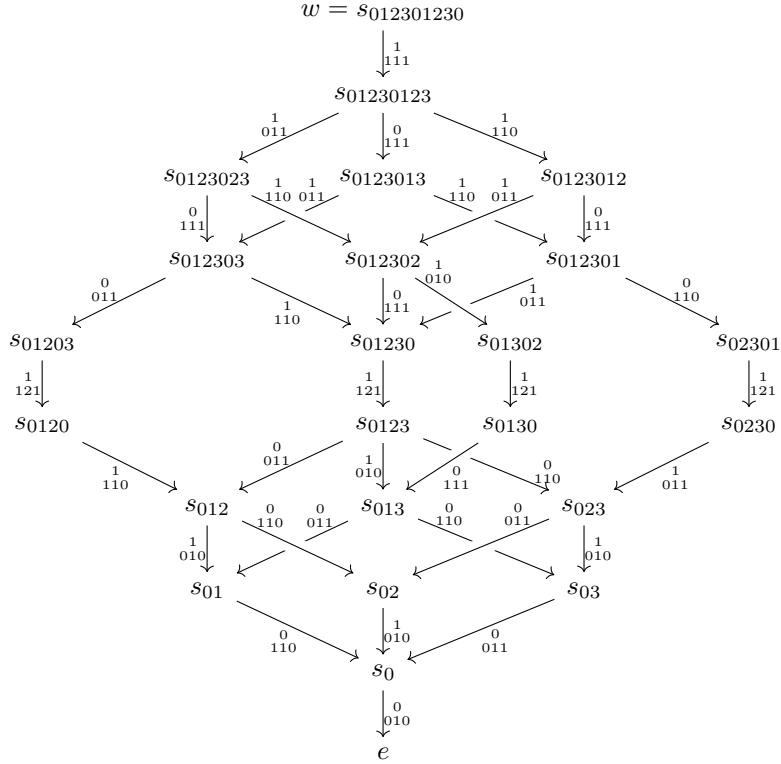
\begin{figure}[h]
  \begin{tikzcd}[every label/.style={auto,font=\tiny, inner sep=0mm}]
    & &[-.5cm] w=s_{012301230}\dar["\substack{1 \\ 111}"] &[-1cm] &[-1.2cm] \\
    & & s_{01230123}\ar[dl, "\substack{1 \\ 011}"'] \dar["\substack{0 \\ 111}"] \ar[drr, "\substack{1 \\ 110}"] \\
    & s_{0123023}\dar["\substack{0 \\ 111}"']  & s_{0123013} \ar[dl, "\substack{1 \\ 011}"', very near start]\ar[drr, "\substack{1 \\ 110}", very near start] && s_{0123012}\dar["\substack{0 \\ 111}"] \ar[dll, "\substack{1 \\ 011}"', very near start, crossing over] \\
    & s_{012303}\ar[dl, "\substack{0 \\ 011}"'] \ar[dr, "\substack{1 \\ 110}"'] & s_{012302} \ar[from=ul, "\substack{1 \\ 110}", crossing over, very near start] \dar["\substack{0 \\ 111}"] & &s_{012301}\ar[dll, "\substack{1 \\ 011}", very near start] \ar[dr, "\substack{0 \\ 110}"] \\
    s_{01203} \dar["\substack{1 \\ 121}"'] & & s_{01230}\dar["\substack{1 \\ 121}"'] & s_{01302} \dar["\substack{1 \\ 121}"] \ar[from=ul, "\substack{1 \\ 010}", crossing over, very near start] & & s_{02301}\dar["\substack{1 \\ 121}"] \\
    s_{0120}\ar[dr, "\substack{1 \\ 110}"'] & & s_{0123} \dar["\substack{1 \\ 010}"'] \ar[dl, "\substack{0 \\ 011}"'] \ar[drr, "\substack{0 \\ 110}", very near end] & s_{0130} \ar[dl, crossing over, "\substack{0 \\ 111}"] & & s_{0230} \ar[dl, "\substack{1 \\ 011}"]\\
    & s_{012} \dar["\substack{1 \\ 010}"']  &
    s_{013}\ar[dl,"\substack{0 \\ 011}"',very near start] \ar[drr,"\substack{0 \\ 110}", very near start] & & s_{023}\ar[dll, "\substack{0 \\ 011}"',very near start, crossing over]\dar["\substack{1 \\ 010}"] \\
    & s_{01} \ar[dr, "\substack{0 \\ 110}"'] & s_{02}\dar["\substack{1 \\ 010}"] \ar[from=ul, "\substack{0 \\ 110}", very near start, crossing over] & & s_{03}\ar[dll, "\substack{0 \\ 011}"] \\
    & & s_0 \dar["\substack{0 \\ 010}"]\\
    & & e
  \end{tikzcd}
  \caption{Root sequences for $w=s_{012301230}$}
  \label{binv:fig:dinv}
  \end{figure}
  For example, the right most path gives a reduced expression $w=s_{012030210}$ with its associated root sequence $\substack{0 \\ 010}, \substack{0 \\ 110}, \substack{1 \\ 010}, \substack{1 \\ 110}, \substack{1 \\ 121}, \substack{0 \\ 011}, \substack{0 \\ 111}, \substack{1 \\ 011}, \substack{1 \\ 111}$
  (here as before, $\substack{1 \\ 011}$ denotes $\al_0 + \al_2 + \al_3$ for example).
\end{example}

\subsection{Bruhat inversions}
A simple root in a root system is \emph{indecomposable in $\Phi^+$} in the  sense that we cannot write it as a positive linear combination of positive roots in a non-trivial way.

We are interested, not in the whole set of positive roots, but in the inversion set $\inv(w)$ for a fixed $w \in W$. It is natural to ask which elements are \emph{indecomposable} in the above sense. This leads to the notion of \emph{Bruhat inversions}.

Let us begin with the property of reflections with respect to inversion roots of $w$.
\begin{proposition}[{\cite[Theorem 5.8]{hum2}, \cite[Theorem 1.4.3]{bb}}]\label{binv:prop:invrefl}
  Let $w$ be an element of $W$ and $\be$ a positive root. Then the following hold:
  \begin{enumerate}
    \item Then $\be \in \inv(w)$ if and only if $\ell(t_\be w) < \ell(w)$ holds.
    \item Let $s_1 s_2 \cdots s_l = w$ be a reduced expression with $s_i = s_{\al_i}$, and let $\be_1,\be_2, \dots, \be_l$ a root sequence associated to it. Then we have
    \[
    t_{\be_i} w = s_1 s_2 \cdots \widehat{s_i} \cdots s_l.
    \]
  \end{enumerate}
\end{proposition}

Next we will define \emph{Bruhat inversions} of elements in $W$. Recall that the \emph{Bruhat order} on $W$ is the transitive closure of the following relation: for every $t \in T$ and $w\in W$ satisfying $\ell(t w) < \ell(w)$, we have $tw \leq w$.
\begin{definition}\label{binv:def:binv}
  Let $w$ be an element of $W$ and $\be$ a positive root. Then $\be$ is a \emph{Bruhat inversion} of $w$ if it satisfies $\ell(t_\be w) = \ell(w) - 1$. We denote by $\Binv(w)$ the set of Bruhat inversions of $w$. We call an element of $\inv(w) \setminus \Binv(w)$ a \emph{non-Bruhat inversion} of $w$.
\end{definition}
By the description in Proposition \ref{binv:prop:invrefl}, the number of Bruhat inversions can be computed as follows: fix a reduced expression $w = s_1 \cdots s_l$ of $w$, then $\#\Binv(w)$ is the number of $i$'s such that deleting $s_i$ from this expression still yields a reduced expression.

A Bruhat inversion is closely related to the covering relation in the Bruhat order as follows.
\begin{proposition}\label{binv:prop:bruhatcover}
  For $w\in W$ and $\be \in \Phi^+$, the following are equivalent:
  \begin{enumerate}
    \item $\be \in \Binv(w)$, that is, $\be$ is a Bruhat inversion of $w$.
    \item $w$ covers $t_\be w$ in the Bruhat order of $w$.
  \end{enumerate}
\end{proposition}
\begin{proof}
  This easily follows from Proposition \ref{binv:prop:invrefl} and the chain property of the Bruhat order, see \cite[Theorem 2.2.6]{bb} for example.
\end{proof}
The following is an example of Bruhat inversions for type A case (we refer the reader to the appendix for the detail).
\begin{example}
   Let $\Phi$ be the standard root system of type $A_n$, and $\al_1,\dots,\al_n$ the simple roots. Then positive roots are of the form $\be_{(i\,\,j)}:=\al_i + \al_{i+1} + \cdots + \al_{j-1}$ for $1 \leq i < j \leq n+1$.
  We can identify $W$ with the symmetric group $S_{n+1}$. For $w$ in $W = S_{n+1}$, we have that $\be_{(i,j)} \in \inv(w)$ if and only if $(i,j)$ is the \emph{classical inversion} of $w$, that is, $i<j$ and $w^{-1}(i) > w^{-1}(j)$ holds.
  Moreover, we can easily see that $\be_{(i,j)} \in \Binv(w)$ if and only if $\be_{(i,j)} \in \inv(w)$ and there is no $i < k < j$ with $\be_{(i,k)},\be_{(k,j)} \in \inv(w)$.
  For example, if $w = 42153 \in S_5$, we have $\inv(w) = \{ \be_{(1,2}, \be_{(1,4)}, \be_{(2,4)}, \be_{(3,4)}, \be_{(3,5)} \}$ and $\Binv(w) = \inv(w) \setminus \{ \be_{(1,4)} \}$.
\end{example}

For a fixed element $w \in W$, we will give a characterization of Bruhat inversions of $w$ among all inversions of $w$.
To do this, we prepare some lemmas.

The following is known as the \emph{lifting property} of the Bruhat order.
\begin{lemma}[{\cite[Proposition 2.2.7]{bb}}]\label{binv:lem:lifting}
  Let $v$ and $w$ be elements in $W$ satisfying $v<w$, and let $s$ be a simple reflection. If $\ell(sv) = \ell(v)+1$ and $\ell(sw) = \ell(w)-1$ holds, then we have $sv < w$ and $v < sw$.
\end{lemma}
By using the lifting property, we can show the following technical lemma.
\begin{lemma}\label{binv:lem:multileft}
  Let $w$ be an element of $W$ and $\be \in \Binv(w)$, and put $t = t_\be \in W$. Then either one of the following holds:
  \begin{enumerate}
    \item $\be$ is a simple root.
    \item There exists a simple reflection $s \in W$ which satisfies the following two conditions:
    \begin{enumerate}
      \item $\ell(s w) = \ell(w) + 1$.
      \item $\ell(s t w) = \ell(tw) + 1$.
    \end{enumerate}
  \end{enumerate}
\end{lemma}
\begin{proof}
  Suppose that (2) does not hold. We will show that $\be$ must be simple.

  Since (2) does not hold, for every simple reflection $s \in W$, we have that either (a) or (b) (or both) fails to hold. We will show the following claim.

  {\bf (Claim)}:
  \emph{There exists some simple reflection $s \in W$ such that {\upshape (a)} does not hold and {\upshape (b)} holds}.

  \emph{Proof of (Claim)}.
  If this is not the case, then (b) does not hold for every simple reflection $s \in W$. This means that $tw$ is the longest element in $W$, which contradicts to $\ell(w) = \ell(tw)+1$. $\qedb$

  Take such a simple reflection $s$. Then we have that $\ell(sw) = \ell(w) -1$ and $\ell(stw) = \ell(tw)+1$ hold, and that $t w < w$ by the assumption.
  Then by Lemma \ref{binv:lem:lifting}, we have that $tw \leq sw < w$ holds. Since $w$ covers $tw$ in the Bruhat order, we must have $tw = sw$, hence $t=s$. Therefore $\be$ is a simple root.
\end{proof}
Now we can show that a Bruhat inversion can be transformed into a simple root:
\begin{proposition}\label{binv:bruhachar0}
  Let $w$ be an element of $W$ and $\be$ an inversion of $w$. Then the following are equivalent:
  \begin{enumerate}
    \item $\be$ is a Bruhat inversion of $w$.
    \item There exists some element $v \in W$ which satisfies the following two conditions:
    \begin{enumerate}
      \item $\ell(vw) = \ell(v) + \ell(w)$ holds.
      \item $v(\be)$ is a simple root.
    \end{enumerate}
  \end{enumerate}
\end{proposition}
\begin{proof}
  (1) $\Rightarrow$ (2):
  Suppose that $\be \in \Binv(w)$ holds.
  If $\be$ is a simple root, then $v=e$ satisfies the conditions (2)(a) and (b).

  From now on, we assume that $\be$ is not simple. Then by Lemma \ref{binv:lem:multileft}, there exists a simple reflection $s$ such that $\ell(s w) = \ell(w) + 1$ and $\ell(s t_\be w)= \ell(w)$ hold. Put $w'=s w$ and $\be':= s(\be)$, then $\be' \in \inv(w')$ by Lemma \ref{binv:lem:invcon}.
  Moreover, $\ell(s t_\be w)= \ell(w)$ implies that $\ell(t_{\be'} w') = \ell(s t_\be s\cdot sw) = \ell(w) = \ell(w')-1$. Thus $\be'$ is a Bruhat inversion of $w'$. If $\be'$ is a simple root, then $v=s$ satisfies the conditions (2)(a) and (b).

  If $\be'$ is not simple, then we can iterate this process by considering $\be'$ and $w'$ instead of $\be$ and $w$. Moreover, this process must stop at some point since otherwise we would have $\ell(w) < \ell(w') < \cdots$, which contradicts to the existence of the longest element in $W$. Therefore, we have that the Bruhat inversion at this point, which can be written as $v(\be)$ for some $v$, is a simple root.

  (2) $\Rightarrow$ (1):
  Since $v(\be)$ is a simple root and $v(\be) \in \inv(vw)$ by Lemma \ref{binv:lem:invcon}, we have $\ell(t_{v(\be)} v w) = \ell(vw) - 1 = \ell(v) + \ell(w) -1$.
  On the other hand, since $t_{v(\be)} v w = v t_\be v^{-1} v w = v t_\be w$ holds, we have $\ell(t_{v(\be)} v w) \leq \ell(v) + \ell(t_\be w)$. Thus $\ell(w)-1 \leq \ell(t_\be w)$.
  Since $\be$ is an inversion of $w$, we have $\ell(t_\be w) \leq \ell(w) -1$ by Proposition \ref{binv:prop:invrefl}. Therefore, we have $\ell(t_\be w )= \ell(w)-1$, that is, $\be$ is a Bruhat inversion of $w$.
\end{proof}

The following is a main result in this section, which gives a characterization of Bruhat inversions.
\begin{theorem}\label{binv:thm:binvchar}
  Let $w$ be an element of the Weyl group $W$ of $\Phi$, and $\be \in \inv(w)$. Then the following are equivalent:
  \begin{enumerate}
    \item $\be$ is a \emph{non-}Bruhat inversion of $w$.
    \item There exists $a_\ga \in \R_{\geq 0}$ for each $\ga \in \inv(w)$ which satisfies the following conditions:
    \begin{itemize}
      \item $\be = \sum_{\ga \in \inv(w)} a_\ga \ga$ holds.
      \item This expression is not of the form $\be = 1 \cdot \beta$, that is, there exists some $\ga \in \inv(w)$ such that $a_\ga \neq \delta_{\be,\ga}$, where $\delta$ is the Kronecker delta.
    \end{itemize}
    \item There exist $\ga_1,\ga_2 \in \inv(w)$ with $\ga_1 \neq \ga_2$ and $a_1, a_2 \in \R_{>0}$ such that $\be = a_1 \ga_1 + a_2 \ga_2$ holds.
    \item There exist $\ga_1, \ga_2 \in \inv(w)$ with $\ga_1 \neq \ga_2$ and $n \in \{1,2,3\}$ such that $\be = (\ga_1 + \ga_2)/n$ holds.
  \end{enumerate}
  Moreover, if $\Phi$ is simply-laced, then the above statements are equivalent to the following:
  \begin{enumerate}[resume]
    \item There exist $\ga_1, \ga_2 \in \inv(w)$ satisfying $\be = \ga_1 + \ga_2$.
  \end{enumerate}
\end{theorem}
\begin{proof}
  The implications (5) $\Rightarrow$ (4) $\Rightarrow$ (3) $\Rightarrow$ (2) are clear.

  (2) $\Rightarrow$ (1):
  Suppose that $\be$ is a Bruhat inversion. Then by Proposition \ref{binv:bruhachar0}, we have an element $v\in W$ such that $\ell(vw) = \ell(v) + \ell(w)$ and that $v(\be)$ is a simple root.
  Note that Lemma \ref{binv:lem:invcon} implies $v(\inv(w)) \subset \Phi^+$.

  Now take $a_\ga \in \R_{\geq 0}$ for $\ga \in \inv(w)$ with $\be = \sum_\ga a_\ga \ga$ as claimed in (2). Then we have
  \[
  v(\be) = \sum_{\ga \in \inv(w)} a_\ga v(\ga).
  \]
  Since $v(\be)$ is a simple root and all $v(\ga)$'s are distinct positive roots, we must have that $a_\ga = \delta_{\be,\ga}$, which contradicts to the condition in (2). Thus $\be$ is not a Bruhat inversion.

  (1) $\Rightarrow$ (4),(5):
  Take any reduced expression $w=s_1 s_2 \cdots s_l$ of $w$, and let $\be_1, \be_2, \dots, \be_l$ be the root sequence associated to it. Then $\be = \be_m$ for some $m$.
  Since we have $t_\be w = s_1 \cdots \widehat{s_m} \cdots s_l$ by Proposition \ref{binv:prop:invrefl} and $\be$ is \emph{not} a Bruhat inversion, the expression $s_1 \cdots \widehat{s_m} \cdots s_l$ is not reduced.

  Take the minimal $j$ such that $s_1 \cdots \widehat{s_m} \cdots s_j$ is \emph{not} a reduced expression. Then by Lemma \ref{binv:lem:exchange}, there exists $i$ such that $s_1 \cdots \widehat{s_m} \cdots s_{j-1} = s_1 \cdots \widehat{s_i} \cdots \widehat{s_m} \cdots s_j$.
  In this situation, we will prove the following claim.

  {\bf (Claim)}:
  \emph{We have $t_\be(\be_j) = -\be_i$}.

  \emph{Proof of (Claim)}.
  First we have the following equation:
  \begin{align*}
    t_\be(\be_j) &= t_{\be_m} (\be_j) \\
    &= (s_1 \cdots s_{m-1} s_m s_{m-1} \cdots s_1)(\be_j)\\
    &= (s_1 \cdots s_{m-1} s_m s_{m-1} \cdots s_1) (s_1 \cdots s_{j-1})(\al_j)\\
    &= (s_1 \cdots \widehat{s_m} \cdots s_{j-1})(\al_j),
  \end{align*}
  where $\al_j$ is a simple root with $s_{\al_j} = s_j$.
  On the other hand, since $s_1 \cdots \widehat{s_m} \cdots s_{j-1} s_j$ is not reduced, Lemma \ref{binv:lem:exchange} implies that $(s_1 \cdots \widehat{s_m} \cdots s_{j-1})(\al_j)$ is a negative root. Thus $t_\be(\be_j)$ is a negative root.

  Now consider the reflection with respect to $t_\be(\be_j)$:
  \begin{align*}
    t_{t_\be(\be_j)} &= t_\be t_{\be_j} t_\be \\
    &= (s_1 \cdots s_m \cdots s_1) (s_1 \cdots s_m \cdots s_j \cdots s_m \cdots s_1) (s_1 \cdots s_m \cdots s_1) \\
    &= s_1 \cdots \widehat{s_m} \cdots s_{j-1} s_j s_{j-1} \cdots \widehat{s_m} \cdots s_1 \\
    &= (s_1 \cdots \widehat{s_i} \cdots \widehat{s_m} \cdots s_j) s_j s_{j-1} \cdots \widehat{s_m} \cdots s_1 \\
    &= s_1 \cdots s_i \cdots s_1 \\
    &= t_{\be_i}.
  \end{align*}
  Thus the reflection with respect to $t_\be(\be_j)$ coincides with that along $\be_i$. Since $t_\be(\be_j)$ is a negative root, we must have $t_\be(\be_j) = - \be_i$. $\qedb$

  Now by (Claim), we have $\be_j - \la \be_j, \be \ra \be = t_\be(\be_j)  = - \be_i$, thus $\la \be_j, \be \ra \be = \be_i + \be_j$.
  Since $\be$, $\be_i$ and $\be_j$ are positive roots, we must have $\la \be_j, \be \ra > 0$.
  We have $\la \be_j, \be \ra \in \{1,2,3\}$ by Proposition \ref{binv:prop:cartanint}, so (4) holds.

  Assume that $\Phi$ is simply-laced. Since $m < j$, we have that $\be_m=\be$ and $\be_j$ are distinct. Thus $\la \be_j, \be \ra =1$ holds by Proposition \ref{binv:prop:cartanint}, so (5) holds.
\end{proof}

By this theorem, a non-Bruhat inversion $\ga$ of $w$ can be written as $\ga = \al + \be$ with $\al,\be \in \inv(w)$ if $\Phi$ is simply-laced. This kind of equation gives a restriction of the relative position of $\al,\be$ and $\ga$ as follows.
\begin{lemma}\label{binv:lem:contfig}
  Suppose that $\Phi$ is simply-laced, and that $\al$ and $\be$ in $\Phi$ satisfies $\ga:=\al+\be \in \Phi$. Then $\la \al,\be\ra = -1$, $\la \al,\ga\ra = \la \be,\ga \ra = 1$ hold, thus they look as follows.
  
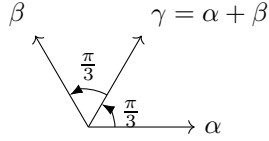
\[
    \begin{tikzpicture}[scale=.7]
      \draw[->] (0,0) -- (0:2) node[right] {$\al$};
      \draw[->] (0,0) -- (60:2) node[above right] {$\ga = \al + \be$};
      \draw[->] (0,0) -- (120:2) node[above left] {$\be$};
      \draw[-Latex] (.5,0) arc (0:60:.5) node[midway, right] {$\frac{\pi}{3}$};
      \draw[-Latex] (60:.7) arc (60:120:.7) node[midway, above] {$\frac{\pi}{3}$};
    \end{tikzpicture}
  \]
\end{lemma}
\begin{proof}
  Clearly $\al,\be, \ga$ are distinct.
  Thus we have that $\la \be,\al \ra, \la \ga,\al \ra \in \{-1,0,1\}$ by Proposition \ref{binv:prop:cartanint}.
  However, we have $\la \ga,\al \ra =\la \al+\be,\al \ra =\la \al,\al\ra + \la \be,\al \ra = 2 + \la \be,\al \ra$. Thus we must have $\la \be,\al \ra = -1$ and $\la \ga,\al \ra = 1$.
  The equation $\la \ga,\be \ra = 1$ can be shown similarly.

  Since $(\al,\be),(\al,\ga) \neq 0$ and $\Phi$ is simply-laced, $\al,\be,\ga$ has the same length. Then it easily follows from $\la \be,\al \ra = -1$ that the angle between $\al$ and $\be$ is $\frac{2}{3}\pi$. Therefore the situation looks like the figure.
\end{proof}

\begin{remark}
  If $\Phi$ is not simply-laced, then a non-Bruhat inversion $\ga$ of $w$ may not be written as a sum of other inversions. For example, let $\Phi$ be the root system of type $B_2$ with $\al$ a short simple root and $\be$ a long simple root:
  
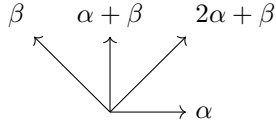
\[
  \begin{tikzpicture}
    \draw[->] (0,0) -- (1,0) node [right] {$\al$};
    \draw[->] (0,0) -- (-1,1) node[above left] {$\be$};
    \draw[->] (0,0) -- (0,1) node[above] {$\al+\be$};
    \draw[->] (0,0) -- (1,1) node[above right] {$2 \al + \be$};
  \end{tikzpicture}
  \]
  Then we have four positive roots. Consider $w = s_\be s_\al s_\be$. Easy computation shows that its root sequence is $\be, \al+\be, 2\al+\be$, thus we have $\inv(w) = \{\be,\al+\be,2\al+\be\}$.
  By using Theorem \ref{binv:thm:binvchar}, one can conclude that $\al+\be$ is not a Bruhat inversion of $w$ since $\al+\be = \be/2 + (2\al+\be)/2$ holds (of course this can be deduced by checking that $s_\be \widehat{s_\al} s_\be$ is not reduced, which is trivial). Nevertheless, we cannot write $\al+\be$ as a sum of other inversions of $w$.
\end{remark}

\section{Brick sequence of a torsion-free class}\label{binv:sec:3}
In this section, we define a brick sequence of a torsion-free class $\FF$, which is associated to a maximal green sequence of $\FF$ (a saturated chain of torsion-free classes between $0$ and $\FF$). In particular, we focus on the relation between brick sequences and simple objects in $\FF$.

Throughout this section, let $\Lambda$ be a finite-dimensional $k$-algebra over a field $k$. We denote by $\mod\Lambda$ the category of finitely generated right $\Lambda$-modules. A module always means a finitely generated right modules.

\subsection{Brick labeling}
We briefly recall the lattice structure of torsion-free classes in $\mod \Lambda$, following \cite{DIRRT}. Note that in \cite{DIRRT}, \emph{torsion classes}, the dual notion of torsion-free classes, were mainly studied, but the same theory works also for torsion-free classes by the standard duality.

A subcategory $\FF$ of $\mod\Lambda$ is a \emph{torsion-free class} if it is closed under extensions and submodules in $\mod\Lambda$. We denote by $\torf\Lambda$ the set of all torsion-free classes in $\mod\Lambda$. Then $\torf\Lambda$ is a poset with respect to inclusion.
Moreover, since intersection of any torsion-free classes is also a torsion-free class, $\torf\Lambda$ is a complete lattice.

For a poset $P$, its \emph{Hasse quiver} $\Hasse P$ is a quiver defined as follows:
\begin{itemize}
  \item A vertex set of $\Hasse P$ is $P$.
  \item We draw a unique arrow $x \to y$ in $\Hasse P$ if $x \cov y$, that is, $x$ covers $y$ in a poset $P$.
\end{itemize}

Now $\Hasse (\torf\Lambda)$ has the additional structure called the \emph{brick labeling}, established in \cite{DIRRT}. Let us introduce some terminologies to state it.
A $\Lambda$-module $M$ in $\mod\Lambda$ is called a \emph{brick} if $\End_\Lambda(M)$ is a division ring.
For a collection $\CC$ of $\Lambda$-modules, we will use the following notations:
\begin{itemize}
  \item $\add \CC$ denotes the category of direct summands of finite direct sums of modules in $\CC$.
  \item $\Sub \CC$ denotes the category of modules $X$ such that there exists an injection from $X$ to a module in $\add\CC$.
  \item $\Filt\CC$ denotes the category of modules $X$ such that there exists a filtration $0 = X_0 \leq X_1 \leq \cdots \leq X_n = X$ of submodules of $X$ satisfying $X_i/X_{i-1} \in \CC$ for each $i$.
  \item $\FFF(\CC)$ denotes the smallest torsion-free class containing $\CC$, or equivalently, $\FFF(\CC) = \Filt(\Sub\CC)$.
  \item $^\perp \CC$ denotes the category of modules $X$ satisfying $\Hom_\Lambda(X,\CC) = 0$.
  \item $\CC^\perp$ denotes the category of modules $X$ satisfying $\Hom_\Lambda(\CC,X) = 0$.
  \item $\ind \CC$ denotes the set of isomorphism classes of indecomposable modules in $\CC$.
  \item $\brick\CC$ denotes the set of isomorphism classes of bricks in $\CC$.
\end{itemize}
For two collections $\CC$ and $\DD$ of $\Lambda$-modules, $\CC * \DD$ denotes the category of modules $X$ such that there exist an exact sequence
\[
\begin{tikzcd}
  0 \rar & C \rar & X \rar & D \rar & 0
\end{tikzcd}
\]
with $C\in \CC$ and $D \in \DD$.
From the classical torsion theory, we have $\mod\Lambda = {}^\perp \FF * \FF$ for a torsion-free class $\FF$ in $\mod\Lambda$, see e.g. \cite[VI.1]{ASS}. It can be easily checked that the operation $*$ is associative, that is, we have $(\CC * \DD) * \EE = \CC * (\DD * \EE)$ holds for collections $\CC$, $\DD$ and $\EE$ of $\Lambda$-modules. Thus we omit parentheses and just write as $\CC * \DD * \EE$.

In \cite{DIRRT}, the following basic observation was established.
\begin{proposition}[{\cite[Theorems 3.3, 3.4]{DIRRT}}]\label{binv:prop:dirrt}
  Let $\GG \subset \FF$ be two torsion-free classes in $\mod\Lambda$. Then there exists an arrow $q\colon \FF \to \GG$ in $\Hasse (\torf\Lambda)$ if and only if $\brick({}^\perp \GG \cap \FF)$ contains exactly one element $B_q$.
  In this case, we have $^\perp\GG \cap \FF = \Filt(B_q)$, $\FF = \Filt(B_q) * \GG$ and $\GG = \FF \cap B_q^\perp$.
\end{proposition}
By this, to each arrow $q$ in $\Hasse(\torf\Lambda)$ we can associate a brick $B_q$, which we call the \emph{label of $q$}. This labeling is called the \emph{brick labeling} of $\torf\Lambda$.

\subsection{Brick sequence associated to a maximal green sequence}
To study the structure of a fixed torsion-free class $\FF$, a \emph{brick sequence} associated to a \emph{maximal green sequence} of $\FF$ plays an important role later. Let us introduce these notions.
\begin{definition}
  Let $\FF$ be a torsion-free class in $\mod\Lambda$. Then a \emph{maximal green sequence of $\FF$} is a finite path in $\Hasse(\torf\Lambda)$ which starts at $\FF$ and ends at $0$. Or equivalently, a mximal green sequence of $\FF$ is a saturated chain in $\torf\Lambda$ of the form $0 = \FF_0 \covd \FF_1 \covd \cdots \covd \FF_l = \FF$.
\end{definition}
Maximal green sequences were introduced by Keller in the context of quiver mutations in cluster algebras, and have been investigated from various viewpoints. We refer the reader to the recent article \cite{DK} and the reference therein for the details on this notion. The above definition is a straightforward generalization of maximal green sequences of abelian categories, which was introduced in \cite{bst}.

To a maximal green sequence of $\FF$, we can associate a sequence of bricks in $\FF$ as follows.
This is analogous to the root sequence associated to a reduced expression in the Weyl group (and actually gives a categorification as we shall see in Proposition \ref{binv:prop:categorify}).
\begin{definition}
  Let $0 = \FF_0 \ot \FF_1 \ot \dots \ot \FF_l = \FF$ be a maximal green sequence of a torsion-free class $\FF$ in $\mod\Lambda$. Then a \emph{brick sequence associated to it} is a sequence $B_1,B_2,\dots,B_l$ of bricks where each $B_i$ is the label of the arrow $\FF_i \ot \FF_{i-1}$ in $\Hasse(\torf\Lambda)$.
  We simply call a brick sequence associated to some maximal green sequence of $\FF$ a \emph{brick sequence of $\FF$}.
\end{definition}
As in the case of root sequences, we take the appearance order of bricks into account. For a fixed $\FF$, brick sequences of $\FF$ heavily depend on the choice of maximal green sequences. In general, the lengths of brick sequences may differ, as the following example shows.

\begin{example}
  Let $Q$ be a quiver $1 \ot 2$. Then $\Hasse(\torf kQ)$ and its brick labeling are as follows:
  
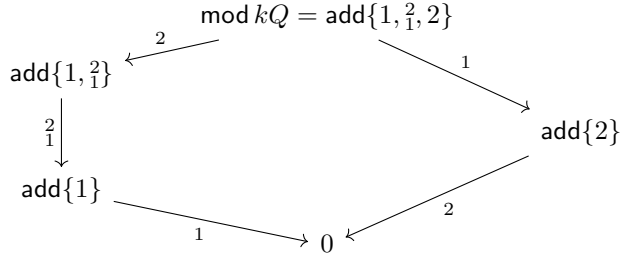
\[
  \begin{tikzcd}[row sep=tiny]
    & \mod kQ = \add\{1,\substack{2 \\ 1}, 2\} \ar[dl, "2"'] \ar[ddr, "1"] \\
    \add \{1, \substack{2 \\ 1} \} \ar[dd, "\substack{2 \\ 1}"'] \\
    & & \add \{2\} \ar[ddl, "2"] \\
    \add \{1\} \ar[dr, "1"'] \\
    & 0
  \end{tikzcd}
  \]
  Here we write composition series to indicate $kQ$-modules. Thus there are exactly two brick sequences of $\mod kQ$, namely, $1,\substack{2 \\ 1}, 2$ and $2,1$.
\end{example}

The fundamental property of a brick sequence is the following.
\begin{theorem}\label{binv:thm:brickseq}
  Let $B_1,\dots,B_l$ be a brick sequence of a torsion-free class $\FF$ in $\mod\Lambda$. Then the following hold.
  \begin{enumerate}
    \item $\Hom_\Lambda(B_j,B_i) = 0$ for $j > i$.
    \item $B_1,\dots,B_l$ are pairwise non-isomorphic.
    \item $\FF = \Filt(B_l) * \Filt(B_{l-1}) * \dots * \Filt(B_2) * \Filt(B_1)$ holds. In particular, we have $\FF = \Filt(B_1,\dots,B_l)$.
  \end{enumerate}
\end{theorem}
\begin{proof}
  Let $0 = \FF_0 \ot \FF_1 \ot \dots \ot \FF_l = \FF$ be a maximal green sequence of $\FF$ which gives a brick sequence $B_1,\dots,B_l$.

  (1)
  By Proposition \ref{binv:prop:dirrt} and the definition of the brick labeling, each $B_i$ is contained in $^\perp \FF_{i-1} \cap \FF_i$.
  Thus $B_j \in {}^\perp \FF_{j-1} \subset {}^\perp \FF_i$ holds since $\FF_{j-1} \supset \FF_i$ for $j > i$. Therefore $\Hom_\Lambda(B_j,B_i) = 0$ by $B_i \in \FF_i$.

  (2)
  This is clear from (1).

  (3)
  We will prove by backward induction on $i$ that ${}^\perp\FF_{i-1} \cap \FF = \Filt(B_l) * \dots * \Filt(B_{i+1}) * \Filt(B_i)$ holds for $1 \leq i \leq l$. For $i=l$, this holds by Proposition \ref{binv:prop:dirrt}.

  Suppose this holds for $i = j+1$ with $1 \leq j < l$, and take $M \in {}^\perp\FF_{j-1} \cap \FF$. Since $^\perp \FF_j * \FF_j = \mod\Lambda$  holds, we have an exact sequence in $\mod\Lambda$
  \[
  \begin{tikzcd}
    0 \rar & T \rar & M \rar & F \rar & 0
  \end{tikzcd}
  \]
  with $T \in {}^\perp\FF_j$ and $F \in \FF_j$.
  Since $\FF$ is closed under submodules and $M \in \FF$,
  we have $T \in {}^\perp\FF_j \cap \FF = \Filt (B_l) * \dots * \Filt(B_{j+1})$ by the induction hypothesis. On the other hand, since ${}^\perp\FF_{j-1}$ is closed under quotients, we have $F \in {}^\perp\FF_{j-1} \cap \FF_j = \Filt (B_j)$. Thus we have $M \in \Filt(B_l) * \dots * \Filt(B_{j+1}) * \Filt(B_j)$.
\end{proof}
We remark that this statement appeared in \cite[Corollary 6.5]{tre} for the case $\FF = \mod\Lambda$, and the filtration given in (3) above is called the \emph{Harder-Narasimhan filtration} associated to a maximal green sequence. Also this filtration was considered in \cite[Theorem 6.8]{tattar} in the setting of quasi-abelian categories. Since torsion-free classes are quasi-abelian, we can apply his result to our setting to obtain this theorem.

\subsection{Simple objects in a torsion-free class}
We introduce the notion of \emph{simple objects} in a torsion-free class.
\begin{definition}
  Let $\FF$ be a torsion-free class in $\mod\Lambda$.
  \begin{enumerate}
    \item $M \in \FF$ is a \emph{simple object in $\FF$} if $M \neq 0$ and for any short exact sequence
    \[
    \begin{tikzcd}
      0 \rar & L \rar & M \rar & N \rar & 0
    \end{tikzcd}
    \]
    of $\Lambda$-modules with $L,N \in \FF$, we have $L = 0$ or $M=0$.
    \item We denote by $\simp\FF$ the set of isomorphism classes of simple objects in $\FF$.
  \end{enumerate}
\end{definition}
Note that whether a given module $M$ is simple object or not depends on the torsion-free class which contains $M$. Originally, the notion of simple objects was introduced and studied in the context of exact categories by several papers such as \cite{eno,bhlr}. Since we are only interested in torsion-free classes, we will not work in full generality.

For a torsion-free class, a simple object can be described by the following property.
\begin{lemma}\label{binv:lem:simpcri}
  Let $\FF$ be a torsion-free class in $\mod\Lambda$ and $M$ a non-zero object in $\FF$. Then the following are equivalent:
  \begin{enumerate}
    \item $M$ is simple in $\FF$.
    \item Every morphism $M \to F$ with $F \in \FF$ is either zero or injective.
  \end{enumerate}
\end{lemma}
\begin{proof}
  (1) $\Rightarrow$ (2):
  Take a non-zero morphism $\varphi \colon M \to F$ with $F \in \FF$. Then we have an exact sequence
  \[
  \begin{tikzcd}
    0 \rar & \ker \varphi \rar & M \rar & \im \varphi \rar & 0
  \end{tikzcd}
  \]
  in $\mod\Lambda$. Since $\im\varphi$ is a submodule of $F \in \FF$, we have $\im\varphi \in \FF$. Similarly $\ker \varphi \in \FF$ holds since so is $M$.
  On the other hand, $\im \varphi \neq 0$ by assumption. Therefore, $\ker \varphi = 0$ holds by the simplicity of $M$.

  (2) $\Rightarrow$ (1):
  Take a short exact sequence
  \[
  \begin{tikzcd}
    0 \rar & L \rar & M \rar["\pi"] & N \rar & 0
  \end{tikzcd}
  \]
  with $L,N \in \FF$. Then by assumption, $\pi$ should be zero or injection. We have $N = 0$ in the former case, and $L = 0$ in the latter. Thus $M$ is a simple object in $\FF$.
\end{proof}

We will investigate the relation between simple objects in $\FF$ and brick sequences of $\FF$.
One of the remarkable property is that simples always appear in every brick sequence of $\FF$:
\begin{proposition}\label{binv:prop:simpseq}
  Let $\FF$ be a torsion-free class in $\mod\Lambda$ and $S$ a simple object in $\FF$. Then $S$ appears exactly once in every brick sequence of it (if exists).
\end{proposition}
\begin{proof}
  Let $B_1,\dots,B_l$ be a brick sequence of $\FF$. Then we have $\FF = \Filt(B_1,\dots,B_l)$ holds by Theorem \ref{binv:thm:brickseq}. This means that $S$ has a filtration consisting of $B_1,\dots,B_l$, but since $S$ is a simple object in $\FF$, clearly $S \iso B_i$ holds for some $i$.
  By Theorem \ref{binv:thm:brickseq}, all the bricks in this brick sequence are pairwise non-isomorphic, thus $S$ must appear exactly once.
\end{proof}
By this, to find a simple object in $\FF$, we only have to work in a fixed brick sequence of $\FF$.
In order to give a criterion for a brick in a brick sequence to be simple, we will introduce some technical condition and lemmas.
\begin{definition}
  Let $M$ be a $\Lambda$-module and $\CC$ a collection of $\Lambda$-modules. Then we say that $\CC$ has the \emph{zero-mono property for $M$} if every morphism $M \to C$ with $C \in \CC$ is either zero or injective in $\mod\Lambda$.
\end{definition}
Then Lemma \ref{binv:lem:simpcri} amounts to that $M$ is simple in a torsion-free class $\FF$ if and only if $\FF$ has the zero-mono property for $M$.

The important advantage of the zero-mono property is that this property is closed under extensions in the following sense:
\begin{lemma}\label{binv:lem:zeromono}
  Let $M$ be a $\Lambda$-module and $\CC,\DD$ collections of $\Lambda$-modules. If $\CC$ and $\DD$ have the zero-mono property for $M$, then so does $\CC * \DD$.
\end{lemma}
\begin{proof}
  Consider any short exact sequence
  \[
  \begin{tikzcd}
    0 \rar & C \rar["\iota"] & X \rar["\pi"] & D \rar & 0
  \end{tikzcd}
  \]
  with $C \in \CC$ and $D \in \DD$.
  Take any morphism $\varphi \colon M \to X$, and it suffices to show that $\varphi$ is either zero or injective.

  Consider the following diagram.
  \[
  \begin{tikzcd}
    & & M \dar["\varphi"] \ar[dl,dashed,"\ov{\varphi}"']\\
    0 \rar & C \rar["\iota"] & X \rar["\pi"] & D \rar & 0
  \end{tikzcd}
  \]
  Since $\DD$ has the zero-mono property for $M$, either $\pi\varphi$ is an injection or $\pi\varphi = 0$. In the former case, $\varphi$ is an injection, so suppose the latter. Then there exists a morphism $\ov{\varphi}\colon M \to C$ with $\varphi = \ov{\varphi} \iota$.
  Since $\CC$ has the zero-mono property for $M$, we have $\ov{\varphi}$ is either zero or injective. Thus $\varphi$ is either zero or injective respectively.
\end{proof}
As a first application of this corollary, we can show the following.
\begin{corollary}
  Let $\FF$ be a torsion-free class in $\mod\Lambda$. Then taking labels gives an injection
  \[
  \{ \text{Arrows in $\Hasse(\torf\Lambda)$ starting at $\FF$} \} \hookrightarrow \simp \FF.
  \]
\end{corollary}
\begin{proof}
  Let $\FF \to \FF'$ be an arrow in $\Hasse(\torf\Lambda)$ with $B$ its label. We will show that $B$ is simple in $\FF$, or equivalently, $\FF$ has the zero-mono property for $B$ by Lemma \ref{binv:lem:simpcri}.

  By Proposition \ref{binv:prop:dirrt}, we have $\FF = \Filt(B) * \FF'$ and $\FF' = \FF \cap B^\perp$. Therefore, according to Lemma \ref{binv:lem:zeromono}, it suffies to show that $B$ and $\FF'$ have the zero-mono property for $B$.
  Clearly $B$ has the zero-mono property for $B$ since $B$ is a brick. On the other hand, since $\FF' \subset B^\perp$, every morphism from $B$ to $F \in \FF'$ should be zero. Thus $\FF'$ has the zero-mono property.
\end{proof}

Another application of Lemma \ref{binv:lem:zeromono} is the following complete description of a brick in a given brick sequence to be simple.
\begin{corollary}\label{binv:cor:simplabel}
  Let $B_1, \dots, B_l$ be a brick sequence of a torsion-free class $\FF$ in $\mod\Lambda$. Then any simple objects in $\FF$ are contained in $\{B_1,\dots,B_l\}$, and the following are equivalent for $1 \leq i \leq l$.
  \begin{enumerate}
    \item $B_i$ is simple in $\FF$.
    \item Every morphism $B_i \to B_j$ is either zero or injective for each $j \neq i$, or equivalently, $j > i$.
    \item $\{B_1,\dots,B_l\}$ has the zero-mono property for $B_i$.
    \item There is no surjection $B_i \defl B_j$ with $i < j$.
  \end{enumerate}
\end{corollary}
\begin{proof}
  Since we have $\Hom_\Lambda(B_i,B_j) = 0$ for $j < i$ by Theorem \ref{binv:thm:brickseq}, the two conditions in (2) are equivalent.

  (1) $\Rightarrow$ (2):
  This is clear by Lemma \ref{binv:lem:simpcri} since $B_j \in \FF$ for every $j$.

  (2) $\Rightarrow$ (3):
  It suffices to recall that $B_i$ has the zero-mono property for $B_i$ since $B_i$ is a brick.

  (3) $\Rightarrow$ (1):
  Lemma \ref{binv:lem:zeromono} shows that $\Filt (B_1, \dots, B_l)$ has the zero-mono property for $B_i$. Since $\FF = \Filt(B_1,\dots,B_l)$ holds by Theorem \ref{binv:thm:brickseq}, we have that $B_i$ is simple in $\FF$ by Lemma \ref{binv:lem:simpcri}.

  (2) $\Rightarrow$ (4):
  This is clear since $B_i$ and $B_j$ with $j > i$ are non-isomorphic by Theorem \ref{binv:thm:brickseq}.

  (4) $\Rightarrow$ (2):
  Suppose that there are some $j > i$ and a map $\varphi\colon B_i \to B_j$ which is neither zero nor injective. Take minimal $j$ with this property. We claim that $\varphi$ is actually a surjection.

  Let $0 = \FF_0 \ot \FF_1 \ot \cdots \ot \FF_l = \FF$ be a maximal green sequence of $\FF$ corresponding to the brick sequence $B_1,\dots,B_l$. Then we have $\FF_j = \Filt(B_j) * \FF_{j-1}$ by Proposition \ref{binv:prop:dirrt}.
  Moreover, since (1)-(3) are equivalent and $B_1,\dots,B_{j-1}$ is a brick sequence of $\FF_{j-1}$, the minimality of $j$ implies that $B_i$ is a simple object in $\FF_{j-1}$, that is, $\FF_{j-1}$ has the zero-mono property for $B_i$.

  We have $\im \varphi \in \FF_j$ since it is a submodule of $B_j \in \FF_j$. Thus by $\FF_j = \Filt(B_j) * \FF_{j-1}$, we have the following exact sequence
  \[
  \begin{tikzcd}
    & & B_i \rar["\varphi"]\dar[twoheadrightarrow, "\pi"] & B_j \\
    0 \rar & X \rar & \im \varphi \rar\ar[ur,hookrightarrow] & F \rar & 0
  \end{tikzcd}
  \]
  with $X \in \Filt (B_j)$ and $F \in \FF_{j-1}$. Suppose that $X = 0$, then $\im \varphi \iso F \in \FF_{j-1}$ holds. Thus $\pi$ is either zero or injective. Since $\varphi \neq 0$, we have that $\pi$ is injective, hence so is $\varphi$, which is a contradiction.
  Thus $X \neq 0$. It follows from $X \in \Filt (B_i)$ and $X \neq 0$ that $\dim B_i \leq \dim X$ holds. On the other hand, since we have injections $X \hookrightarrow \im \varphi \hookrightarrow B_i$, we have $\dim X \leq \dim (\im \varphi) \leq \dim B_i$.
  It follows that $\dim (\im \varphi) = \dim B_i$, hence the inclusion map $\im \varphi \hookrightarrow B_i$ is an isomorphism. Therefore, $\varphi$ is a surjection.
\end{proof}
Thus, the problem to determine simple objects in a torsion-free class $\FF$ can be solved by the following process in principle.
\begin{enumerate}
  \item Find and fix a maximal green sequence of $\FF$ (if exists).
  \item Compute a brick sequence $B_1, \dots, B_l$ associated to it.
  \item For each $i$, check whether every morphism $B_i \to B_j$ with $i<j$ is either zero or injective (or equivalently, check whether there is no surjection $B_i \defl B_j$ with $i < j$).
  \item $\simp \FF$ consists of $B_i$ with such a property.
\end{enumerate}
Probably the most non-trivial part of the above computation is (1) and (2). If $\FF$ is functorially finite, then (1) can be computed by using mutations of support $\tau$-tilting modules in \cite{AIR}, and (2) can be computed by using the description of brick labels associated to mutations due to \cite{asai}.
\begin{example}\label{binv:ex:dirrt}
  We will give two examples of computation of simple objects by considering algebras borrowed from \cite[Example 1.14]{DIRRT}. Let $\Lambda_1$ and $\Lambda_2$ be $k$-algebras defined by
  \[
  \Lambda_1 = k \left( \begin{tikzcd} 1 \ar[loop left, "u"] \rar & 2 \end{tikzcd}\right) / (u^2) \quad \text{and} \quad
  \Lambda_2 = k \left( \begin{tikzcd} 1 \rar & 2 \ar[loop right, "v"] \end{tikzcd}\right) / (v^2).
  \]
  Then $\torf\Lambda_i$ and their brick labeling for $i=1,2$ is as follows:
  
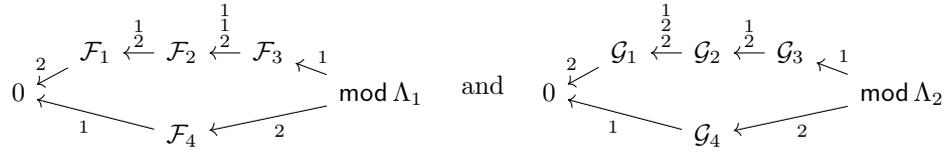
\[
  \begin{tikzcd}[row sep=0mm, column sep = small]
    & \FF_1 \ar[ld, "2"'] & \FF_2 \lar["\substack{1 \\ 2}"'] & \FF_3 \lar["\substack{1 \\ 1 \\ 2}"'] \\
    0 & & & & \mod\Lambda_1 \ar[lu, "1"'] \ar[lld, "2"] \\
     & & \FF_4 \ar[llu, "1"]
  \end{tikzcd}
  \quad \text{and} \quad
  \begin{tikzcd}[row sep=0mm, column sep = small]
    & \GG_1 \ar[ld, "2"'] & \GG_2 \lar["\substack{1 \\ 2 \\ 2}"'] & \GG_3 \lar["\substack{1 \\ 2}"'] \\
    0 & & & & \mod\Lambda_2 \ar[lu, "1"'] \ar[lld, "2"] \\
     & & \GG_4 \ar[llu, "1"]
  \end{tikzcd}
  \]
  We omit the description of each torsion-free class. By using Corollary \ref{binv:cor:simplabel}, one can obtain the following table of simple objects.
  \begin{table}[h]
    \begin{tabular}{C|C}
      \text{torsion-free class in $\mod\Lambda_1$} & \text{simple objects}   \\ \hline \hline
      0 & \varnothing \\ \hline
      \FF_1 & 2 \\ \hline
      \FF_2 & 2, \substack{1 \\ 2}\\ \hline
      \FF_3 & 2, \substack{1 \\ 2}, \substack{1 \\ 1 \\ 2} \\ \hline
      \FF_4 & 1 \\ \hline
      \mod\Lambda_1 & 1,2
    \end{tabular}
    \begin{tabular}{C|C}
      \text{torsion-free class in $\mod\Lambda_1$} & \text{simple objects} \\ \hline \hline
      0 & \varnothing \\ \hline
      \GG_1 & 2 \\ \hline
      \GG_2 & 2, \substack{1 \\ 2 \\ 2}\\ \hline
      \GG_3 & 2, \substack{1 \\ 2}\\ \hline
      \GG_4 & 1 \\ \hline
      \mod\Lambda_1 & 1,2
    \end{tabular}
    \caption{Simple objects in torsion-free classes in $\mod \Lambda_1$ and $\mod\Lambda_2$}
  \end{table}
  Note that $\torf\Lambda_1$ and $\torf\Lambda_2$ are isomorphic as posets, and $\FF_2$ corresponds to $\GG_2$. However, the number of simple objects in $\FF_2$ differs from that in $\GG_2$. This shows that to determine (the number of) simple objects, the poset structure of $\torf\Lambda$ is not enough, and we indeed have to compute brick labels and homomorphisms between labels in general.
\end{example}

Fortunately, in the case of preprojective algebras (or path algebras) of Dynkin type, we can make use of root-theoretical properties developed in the previous section to determine simples, as we shall see in the next section.

We end this section by the following small lemma, which we need later.
\begin{lemma}\label{binv:lem:brickappear}
  Let $\FF$ be a torsion-free class in $\mod\Lambda$ such that the interval $[0,\FF]$ in $\torf\Lambda$ is finite. Then every brick $B$ in $\FF$ appears at least one brick sequence of $\FF$.
\end{lemma}
\begin{proof}
  Consider $\FFF(B)$, the smallest torsion-free class containing $\FF$. Then $\FFF(B) \subset \FF$ holds, that is, $\FFF(B) \in [0,\FF]$.

  Since $[0,\FF]$ is a finite lattice, clearly there exists a maximal green sequence of $\FF$ which is of the form $0 \ot \dots \ot \FFF(B) \ot \dots \ot \FF$. In \cite[Theorem 3.4]{DIRRT}, it was shown that there is only one arrow which starts at $\FFF(B)$ in $\Hasse (\torf\Lambda)$, and that its label is $B$.
  Therefore, this maximal sequence gives the desired brick sequence.
\end{proof}

\section{Torsion-free classes over preprojective algebras of Dynkin type}\label{binv:sec:4}
In this section, we will classify simples in torsion-free classes in preprojective algebras of Dynkin type, by using root systems and brick sequences.

\subsection{Notation and preliminary results}\label{binv:sec:notation}
Let us briefly recall the definition of preprojective algebras of Dynkin type, with emphasis on their relation to root systems.

Let $\Phi$ be the simply-laced root system of the Dynkin type $X$, namely, $X \in \{ A_n, D_n, E_6, E_7,E_8\}$. We denote by $W$ its Weyl group, and we fix a choice of simple roots of $\Phi$.

Let $Q = (Q_0,Q_1)$ be a Dynkin quiver of the same type $X$, that is, $Q$ is a quiver whose underlying graph is the Dynkin diagram of type $X$. Then we may identify $Q_0$ with the index set of simple roots of $\Phi$. For $i \in Q_0$, we denote by $\al_i$ the simple root of $\Phi$ corresponding to the vertex $i$ in the Dynkin diagram, and by $s_i = s_{\al_i}\in W$ the simple reflection with respect to $\al_i$.

Let $\ov{Q}$ be a \emph{double quiver} of $Q$, which is obtained from $Q$ by adding an arrow $a^* \colon j \to i$ for each arrow $a\colon i \to j$ in $Q$.
The \emph{preprojective algebra} of $Q$ is defined by
\[
\Pi = \Pi_\Phi := \left.k\ov{Q} \middle/ \left(\sum_{a \in Q_1} a a^* - a^* a \right) \right. .
\]
It is known that $\Pi$ is a finite-dimensional $k$-algebra, and it only depends on $\Phi$ and does not depend on the choice of $Q$ (the choice of orientations of the Dynkin diagram) up to isomorphism. Thus we call it a \emph{preprojective algebra of $\Phi$}.

It is convenient to consider the dimension vector of $\Pi$-modules inside the ambient space of $\Phi$. Let $V$ be the ambient space of $\Phi$, then $V$ has a basis $\{ \al_i \, | \, i \in Q_0 \}$ as a $\R$-vector space.
Define a group homomorphism $\udim \colon \KKK_0(\mod\Pi) \to V$ by $[S_i] \mapsto \al_i$, where $\KKK_0(\mod\Pi)$ denotes the Grothendieck group of $\mod \Pi$ and $S_i$ denotes the simple $\Pi$-module corresponding to the vertex $i \in Q_0$.
We simply write $\udim M:= \udim [M]$ for $M \in \mod\Pi$.

To sum up, \emph{throughout this section, we keep the following notation:}
\begin{itemize}
  \item $\Phi$ is a simply-laced root system of Dynkin type $X$ in the ambient space $V$.
  \item $Q$ is a Dynkin quiver of the same type $X$.
  \item $W$ is the Weyl group of $\Phi$.
  \item $\Pi$ is the preprojective algebra of $\Phi$ (or $Q$).
  \item $S_i$ is the simple $\Pi$-module corresponding to the vertex $i \in Q_0$.
  \item $\udim \colon \KKK_0(\mod\Pi) \to V$ is a map defined by $[S_i] \mapsto \al_i$ for $i \in Q_0$.
\end{itemize}

Torsion-free classes in $\mod \Pi$ was completely classified by Mizuno in \cite{mizuno}, and they are in bijection with elements in $W$. Let us briefly explain his result in our context.

For a vertex $i \in Q_0$, we denote by $e_i$ the corresponding idempotent of $\Pi$. We denote by $I_i$ the two-sided ideal of $\Pi$ generated by $1-e_i$.
For an element $w \in W$, we can define a two-sided ideal $I_w$ of $\Pi$ as follows: Take any reduced expression $w = s_{u_1} s_{u_2} \cdots s_{u_l}$ of $w$ in $W$. Then $I(w)$ is defined by
\[
I(w) = I_{u_l \dots u_2 u_1} := I_{u_l} I_{u_{l-1}} \cdots I_{u_2} I_{u_1}.
\]
This construction does not depend on the choice of reduced expressions of $w$ by \cite[Theorem III.1.9]{BIRS}.

The following result of Mizuno gives the key connection between the representation theory of $\Pi$ and its root system $\Phi$.
\begin{theorem}[{\cite[Theorem 2.30]{mizuno}}]\label{binv:thm:mizuno}
  Let $\Phi$ be a simply-laced root system of Dynkin type, $W$ its Weyl group and $\Pi$ its preprojective algebra.
  \begin{enumerate}
    \item $\FF(w):= \Sub (\Pi/I(w))$ is a torsion-free class in $\mod\Pi$.
    \item The map $w \mapsto \FF(w)$ gives a bijection
    \[
    W \xrightarrow{\sim} \torf\Pi.
    \]
    Moreover, this bijection is actually an isomorphism of finite lattices, where we endow $W$ with the right weak order.
  \end{enumerate}
\end{theorem}

\subsection{Brick sequences in preprojective algebras}
We begin with studying the relation between brick sequences of $\FF(w)$ and root sequences of $w$.

By Theorem \ref{binv:thm:mizuno}, we can identify a maximal green sequence of $\FF(w)$ with a saturated chain in the interval $[e,w]$ in $(W,\leq_R)$, which in turn is identified with a particular choice of reduced expression of $w$ (see e.g. \cite[Proposition 3.1.2]{bb}). In this way, we can talk about a \emph{brick sequence of $\FF(w)$ associated to a reduced expression of $w$}.
More precisely, let $w = s_{u_1} \cdots s_{u_l}$ be a reduced expression of $w \in W$. Then we have a maximal green sequence $0 = \FF(e) \ot \FF(s_{u_1}) \ot \FF(s_{u_1}s_{u_2}) \ot \cdots \ot \FF(s_{u_1}\cdots s_{u_l}) = \FF(w)$ of $\FF(w)$, thus we obtain the corresponding brick sequence of $\FF(w)$.
We remark that this sequence coincides with a sequence of \emph{layer modules} considered in \cite{AIRT}.

By a lattice isomorphism in Theorem \ref{binv:thm:mizuno}, an arrow in $\Hasse(\torf\Pi)$ is of the form $\FF(w) \ot \FF(w s_i)$ for some $w \in W$ and $i \in Q_0$ satisfying $\ell(w) < \ell(w s_i)$. In this case, we have $I(w s_i) \subset I(w)$, and the following holds.
\begin{proposition}[{\cite[Theorem 4.1]{IRRT}}]
  The brick label of $\FF(w) \ot \FF(w s_i)$ is given by $I(w)/I(w s_i)$.
\end{proposition}
Next we will compute dimension vectors of brick labels following \cite{AIRT}.
\begin{proposition}\label{binv:prop:categorify}
  Let $w = s_{u_1} \dots s_{u_l}$ be a reduced expression of $w \in W$ and $B_1,B_2,\dots,B_l$ its associated brick sequence of $\FF(w)$. Then the following equality holds in $V$ for $ 1 \leq m \leq l$:
  \[
  \udim B_m = s_{u_1} \cdots s_{u_{m-1}}(\al_{u_m}).
  \]
  In particular, $\{\udim B_1, \dots, \udim B_l\} = \inv(w) \subset \Phi^+$ holds, and $\udim B_1,\udim B_2, \dots, \udim B_l$ is a root sequence associated to the above reduced expression of $w$, defined in Definition \ref{binv:def:rootseq}.
\end{proposition}
\begin{proof}
  In \cite[Theorem 2.7]{AIRT}, it was shown that
  \[
  [B_m] = [I(s_{u_{m-1}} \cdots s_{u_1})/I(s_{u_m} \cdots s_{u_1})]= R_{u_1} \cdots R_{u_{m-1}} [S_{u_j}]
  \]
  holds in $\KKK_0(\mod\Pi)$, where $R_i \colon \KKK_0(\mod\Pi) \to \KKK_0(\mod\Pi)$ for $i \in Q_0$ is a group homomorphism defined by
  \[
  R_i([S_j]) := [S_j] - (2 \delta_{ij} - m_{ij}) [S_i]
  \]
  on the free basis $\{ [S_j] \, | \, j \in Q_0 \}$ of $\KKK_0(\mod\Pi)$. Here $\delta$ is the Kronecker delta, and $m_{ij}$ is the number of edges in the Dynkin diagram of $\Phi$ which connect $i$ and $j$.

  Therefore, it suffices to check that the following diagram commutes:
  \[
  \begin{tikzcd}
    \KKK_0(\mod\Pi) \rar["R_i"] \dar["\udim"]& \KKK_0(\mod\Pi) \dar["\udim"] \\
    V \rar["s_i"] & V
  \end{tikzcd}
  \]
  We only have to check on the basis $[S_j]$ for $j \in Q_0$. This follows from the following equality:
  \begin{equation}\label{binv:fig:cartanint}
  2 \delta_{ij} - m_{ij} = \la \al_j,\al_i \ra =
  \begin{cases}
    2 & \text{if $i = j$,} \\
    0 & \text{if $i$ and $j$ are not connected by an arrow in $Q$,} \\
    -1 & \text{if $i$ and $j$ are connected by an arrow in $Q$},
  \end{cases}
  \end{equation}
  which can be checked directly.
\end{proof}
Therefore, a brick sequence of $\FF(w)$ serves as a categorification of a root sequence of $w$. Note that different bricks may have the same dimension vector.
\begin{example}
  Consider an element $w$ in Example \ref{binv:ex:ainv}. Then Figure \ref{binv:fig:ainvbrick} is the Hasse quiver of the interval $[0,\FF(w)]$ with its brick labels.
  \begin{figure}[h]
  \begin{tikzcd}[column sep=tiny, row sep=small]
    &[.5cm] & \FF(w) \ar[dl, "\substack{\; 3\\2}"'] \ar[dr, "1"] \\
    & \bullet \dar["\substack{\;\; 3 \\ \; 2 \\ 1}"] \ar[dl,"2"'] & & \bullet \dar["\substack{1\; 3\\2}"] \\
    \bullet \dar["\substack{\;\; 3 \\ \; 2 \\ 1}"'] & \bullet \ar[dl, "2"] \ar[dr, "1"'] & & \bullet \dar["\substack{1\\ \; 2}"] \ar[dl,"\substack{\; 3\\2}"']\\
    \bullet \dar["\substack{\; 2\\1}"'] & & \bullet \dar["\substack{1\\ \; 2}"'] & \bullet \ar[dl, "\substack{\; 3\\2}"]\\
    \bullet \ar[rd, "1"'] & & \bullet \ar[ld, "2"]&  \\
    & 0
  \end{tikzcd}
  \caption{Brick sequences of $w=s_{12312}$}
  \label{binv:fig:ainvbrick}
  \end{figure}
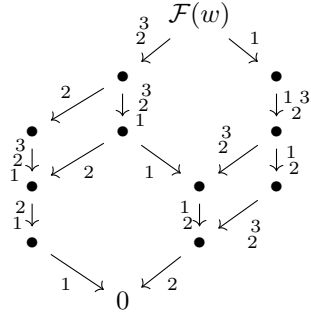
  For example, two bricks $\substack{1\\ \; 2}$ and $\substack{\;2 \\ 1}$ are non-isomorphic but have the same dimension vector $\al_1 + \al_2$.

  The same computation can be done for Example \ref{binv:ex:dinv}, and actually Figure \ref{binv:fig:ainvbrick} can be seen as a Hasse quiver of $[0,\FF(w)]$ with dimension vectors of its brick labels.
\end{example}

\begin{corollary}\label{binv:cor:diminv}
  Let $w$ be an element of $W$, then the following hold.
  \begin{enumerate}
    \item For any $M \in \FF(w)$, we have
    \[
    \udim M = \sum_{\be \in \inv(w)} n_\be \be,
    \]
    where $n_\be$ is some non-negative integer for each $\be \in \inv(w)$.
    \item For any $B \in \brick \FF(w)$, we have $\udim B \in \inv(w)$. Thus $\udim$ induces a surjection $\udim \colon \brick\FF(w) \to \inv(w)$.
  \end{enumerate}
\end{corollary}
\begin{proof}
  (1)
  Let $B_1,\dots, B_l$ be a brick sequence of $\FF(w)$. Then by Proposition \ref{binv:prop:categorify}, we have $\inv(w) = \{ \udim B_1,\dots,\udim B_l\}$.

  Let $M \in \FF(w)$. Then Theorem \ref{binv:thm:brickseq} implies that $M \in \Filt(B_1,\dots, B_l)$. Thus by taking the dimension vector, the assertion immediately follows.

  (2)
  Let $B$ be a brick in $\FF(w)$. Then Lemma \ref{binv:lem:brickappear} implies that there is a brick sequence of $\FF(w)$ which contains $B$, since $[0,\FF(w)] \subset \torf\Pi \iso W$ is a finite lattice. Now Proposition \ref{binv:prop:categorify} implies that $\udim B \in \inv(w)$.
  Thus the map $\udim \colon \brick\FF(w) \to \inv(w)$ is well-defined.
  Moreover, it is clearly surjective by Proposition \ref{binv:prop:categorify}.
\end{proof}

\subsection{Simple objects versus Bruhat inversions}
By Proposition \ref{binv:prop:simpseq}, to find simple objects in $\FF(w)$, it suffices to check whether $B_i$ is simple or not for a fixed brick sequence $B_1, \dots, B_l$ of $\FF(w)$.
For this, the \emph{homological lemma} due to Crawley-Boevey is useful.
To state this, let us introduce the symmetric bilinear form $\la -,-\ra_\Pi$ on $\KKK_0(\mod\Pi)$ defined by
\[
\la \sum_i a_i[S_i], \sum_j b_j[S_j] \ra_\Pi = 2 \sum_{i \in Q_0} a_i b_i - \sum_{i \to j \in Q_1} (a_i b_j + a_j b_i).
\]
This coincides with the standard homological symmetric bilinear form associated to the quiver $Q$ (or twice of it, see e.g. \cite[VII.4]{ASS}). It can also be interpreted as the restriction of the Euler form of the preprojective algebra $\widehat{\Pi}$ of the extended Dynkin type corresponding to $\Pi$, see e.g. \cite[Section 3]{IRRT}.
Then we have the following formula due to Crawley-Boevey.
\begin{lemma}[{\cite[Lemma 1]{CB}}]\label{binv:lem:cbhom}
  Let $M,N \in \mod \Pi$. Then we have the following equation:
  \[
  \la [M],[N] \ra_\Pi = \dim \Hom_\Pi(M,N) + \dim \Hom_\Pi(N,M) - \dim \Ext_\Pi^1(M,N).
  \]
\end{lemma}
We remark that this formula can be shown easily by using the 2-Calabi-Yau property of the preprojective algebra $\widehat{\Pi}$ of the extended Dynkin type.

We can check the compatibility of the bilinear form on $\KKK_0(\mod\Pi)$ and the value $\la \al,\be \ra$ in $\al,\be \in V$ defined in Section 2 as follows.
\begin{lemma}\label{binv:lem:bilincomp}
  Suppose that $M,N \in \mod \Pi$ satisfy $\udim M,\udim N \in \Phi^+$. Then we have
  \[
  \la [M],[N] \ra_\Pi = \la \udim M,\udim N \ra.
  \]
\end{lemma}
\begin{proof}
  Since $\udim M$ and $\udim N$ are positive roots, we can write as $\udim M = \sum_i m_i \al_i$ and $\udim N = \sum_j n_j \al_j$ for non-negative integers $m_i,n_j$. Then we have
  \begin{align*}
    \la \udim M,\udim N \ra & = \la \sum_i m_i \al_i, \sum_j n_j \al_j \ra \\
    &= \sum_i \left( m_i \la \al_i,\sum_j n_j \al_j \ra \right) \\
    & = \sum_i \left( m_i \la \sum_j n_j \al_j, \al_i \ra \right) \\
    &= \sum_i \sum_j m_i n_j \la \al_j,\al_i \ra.
  \end{align*}
  The first equality follows from definition, the second and the last follow since $\la -,-\ra$ is linear with respect to the first variable, and the third follows since $\Phi$ is simply-laced and both vectors inside $\la-,-\ra$ are roots.
  By using the equation (\ref{binv:fig:cartanint}), we can compute this as follows:
  \[
  \sum_i \sum_j m_i n_j \la \al_j,\al_i \ra = 2 \sum_{i \in Q_0} m_i m_i - \sum_{i \to j \in Q} (m_i n_j + m_i n_j)
  \]
  Thus we have the assertion.
\end{proof}

By using this, we can show the following main result of this paper.
\begin{theorem}\label{binv:thm:main}
  Let $w$ be an element of $W$ and $B_1,\dots, B_l$ a brick sequence of $\FF(w)$. Then the following are equivalent for $1 \leq m \leq l$
  \begin{enumerate}
    \item $\udim B_m \in \Binv(w)$ holds.
    \item $B_m$ is a simple object in $\FF(w)$.
  \end{enumerate}
\end{theorem}
\begin{proof}
  (1) $\Rightarrow$ (2)
  Suppose that $B_m$ is not a simple object in $\FF(w)$. Then there exists an exact sequence
  \[
  \begin{tikzcd}
    0 \rar & L \rar & B_m \rar & N \rar & 0
  \end{tikzcd}
  \]
   with $L,N \neq 0$. By applying $\udim$, we obtain
   \[
   \udim B_m = \udim L + \udim N
   \]
   with $\udim L, \udim N \neq 0$. By Corollary \ref{binv:cor:diminv}, both $\udim L$ and $\udim N$ are non-negative integer linear combinations of inversions of $w$ such that at least one of the coefficients should be strictly positive.
   Thus Theorem \ref{binv:thm:binvchar} implies that $\udim B_m$ is a non-Bruhat inversion.

   (2) $\Rightarrow$ (1):
   Suppose that $\udim B_m$ is a non-Bruhat inversion of $w$. We will use Lemma \ref{binv:lem:simpcri} to show that $B_m$ is not a simple object in $\FF(w)$.

   By Theorem \ref{binv:thm:binvchar}, there exists some $\al,\be \in \inv(w)$ such that $\al + \be = \udim B_m$ since $\Phi$ is simply-laced. Moreover, these satisfy $\la \udim B_m,\al \ra = \la \udim B_m,\be \ra = 1$ by Lemma \ref{binv:lem:contfig}.

  On the other hand, since $\{\udim B_1, \dots, \udim B_l \} = \inv(w)$ by Proposition \ref{binv:prop:categorify}, there are $i$ and $j$ such that $\udim B_i = \al$ and $\udim B_j = \be$. By exchanging $\al$ and $\be$ if necessary, we may assume that $i < j$.
   Moreover, since $\udim B_1,\dots, \udim B_l$ is a root sequence and $\udim B_m = \udim B_i + \udim B_j$ holds, Theorem \ref{binv:thm:rootseqchar} implies $i < m < j$.

   To summarize, we have found $i$ and $j$ with $i < m < j$ such that $\udim B_m = \udim B_i + \udim B_j$ and $\la \udim B_m,\udim B_j \ra = 1$ hold. Then Lemma \ref{binv:lem:bilincomp} implies
   \[
   \la [B_m],[B_j] \ra_\Pi = \la \udim B_m,\udim B_j \ra = 1.
   \]
   By combining this with Lemma \ref{binv:lem:cbhom}, we have
   \[
   \dim \Hom_\Pi(B_m,B_j) + \dim\Hom_\Pi(B_j,B_m) = 1 + \dim \Ext_\Pi^1(B_j,B_m) \geq 1.
   \]
   On the other hand, since $m < j$, we must have $\Hom_\Pi(B_j,B_m) = 0$ by Theorem \ref{binv:thm:brickseq}. Therefore, we have $\dim \Hom_\Pi(B_m,B_j) \geq 1$, that is, $\Hom_\Pi(B_m,B_j) \neq 0$.

   Take any non-zero morphism $\varphi \colon B_m \to B_j$. Since $\udim B_m = \udim B_k + \udim B_m$, we have $\dim B_m > \dim B_k$, hence $\varphi$ cannot be an injection. Therefore Lemma \ref{binv:lem:simpcri} implies that $B_m$ is not a simple object in $\FF(w)$.
\end{proof}

In conclusion, we have the following classification of simple objects.
\begin{corollary}\label{binv:cor:main}
  Let $w$ be an element of $W$. Then the following hold.
  \begin{enumerate}
    \item A brick $B$ in $\FF(w)$ is a simple object in $\FF(w)$ if and only if $\udim B \in \Binv(w)$ holds.
    \item The map $\udim \colon \brick\FF(w) \to \inv(w)$ in Corollary \ref{binv:cor:diminv} restricts to a bijection
    \[
    \udim \colon \simp\FF(w) \xrightarrow{\sim} \Binv(w).
    \]
    In other words, simple objects in $\FF(w)$ bijectively correspond to Bruhat inversions of $w$ by taking dimension vectors.
  \end{enumerate}\end{corollary}
\begin{proof}
  (1)
  Lemma \ref{binv:lem:brickappear} implies that $B$ appears in some brick sequence of $\FF(w)$. Then (1) is obvious by Theorem \ref{binv:thm:main}.

  (2)
  Fix a brick sequence $B_1, \dots, B_l$ of $\FF(w)$. If $S$ is a simple object in $\FF(w)$, then $S\iso B_j$ for some $j$ by Proposition \ref{binv:prop:simpseq}. Thus $\udim S \in \Binv(w)$ holds by Theorem \ref{binv:thm:main}. Thus we obtain a map $\udim \colon \simp\FF \to \Binv(w)$.

  We claim that this map is a bijection. Let $\ga \in \Binv(w)$. Then since $\{\udim B_1,\dots,\udim B_l\} = \inv(w) \supset \Binv(w)$ by Proposition \ref{binv:prop:categorify}, there exists $j$ with $\udim B_j = \ga$. This $B_j$ is simple in $\FF(w)$ by Theorem \ref{binv:thm:main}.
  Thus $\udim \colon \simp\FF \to \Binv(w)$ is surjective.

  On the other hand, let $S$ and $S'$ be simple objects in $\FF(w)$ satisfying $\udim S = \udim S'$. Then by the above argument, we have $S \iso B_i$ and $S' \iso B_j$ for some $i$ and $j$, hence $\udim B_i = \udim B_j$ holds.
  Since elements in $\{\udim B_1, \dots, \udim B_l\} = \inv(w)$ are pairwise distinct by Proposition \ref{binv:prop:invdist}, we have $i = j$, which shows $S \iso S'$. Thus $\udim \colon \simp\FF \to \Binv(w)$ is injective.
\end{proof}

\subsection{Characterization of the Jordan-H\"older property}
Next we will characterize when $\FF(w)$ satisfies the Jordan-H\"older property in the sense of \cite{eno} in terms of the combinatorics of $w$.
Let us recall some related definitions and results from \cite{eno}.
\begin{definition}\label{binv:def:jhp}
  Let $\FF$ be a torsion-free class in $\mod\Lambda$ for a finite-dimensional $k$-algebra $\Lambda$.
  \begin{enumerate}
    \item For $M$ in $\FF$, a \emph{composition series of $M$ in $\FF$} is a series of submodules of $M$
    \[
    0 = M_0 \subset M_1 \subset \dots \subset M_m
    \]
    such that $M_i/M_{i-1}$ is a simple object in $\FF$ for each $i$.
    \item For $M$ in $\FF$, let $0 = M_0 \subset \cdots \subset M_m = M$ and $0 = M'_1 \subset \cdots \subset M'_n = M$ be two composition series of $M$ in $\FF$.
    We say that these are \emph{equivalent} if $m = n$ holds and there exists a permutation $\sigma$ of the set $\{ 1,2, \dots,n\}$ such that $M_i/M_{i-1} \iso M'_{\sigma(i)}/M'_{\sigma(i)-1}$ holds for each $i$.
    \item We say that $\FF$ satisfies the \emph{Jordan-H\"older property}, abbreviated by \emph{(JHP)}, if any composition series of $M$ are equivalent for every object $M$ in $\FF$.
  \end{enumerate}
\end{definition}
In \cite[Theorem 5.10]{eno}, the author gives a numerical criterion for (JHP). To rephrase his result in our context, we introduce the \emph{support} of modules or torsion-free classes.
\begin{definition}
  Let $\Lambda$ be a finite-dimensional $k$-algebra and $\simp(\mod\Lambda)$ the set of isomorphism classes of simple $\Lambda$-modules.
  \begin{enumerate}
    \item For a module $M$, the \emph{support of $M$} is a set of simple $\Lambda$-modules defined by
    \[
    \supp M := \{ S \in \simp(\mod\Lambda) \, | \, \text{$S$ is a composition factor of $M$} \}.
    \]
    \item For a collection $\CC$ of modules, the \emph{support of $\CC$} is a set of simple $\Lambda$-modules defined by
    \[
    \supp \CC := \bigcup_{M \in \CC} \supp M.
    \]
  \end{enumerate}
\end{definition}
Then the following gives a numerical criterion for (JHP). Here for a set $A$, we denote by $\Z^{(A)}$ the free abelian group with basis $A$. We simply write $\Z^A := \Z^{(A)}$ if $A$ is a finite set.
\begin{theorem}\label{binv:thm:jhpchar}
  Let $\FF$ be a torsion-free class in $\mod\Lambda$ for a finite-dimensional $k$-algebra $\Lambda$. Suppose that $\FF = \Sub M$ holds for some $M \in \FF$. Then the following are equivalent:
  \begin{enumerate}
    \item $\FF$ satisfies (JHP).
    \item The natural map $\Z^{(\simp\FF)} \to \KKK_0(\mod\Lambda)$ which sends $M \in \simp\FF$ to $[M]$ is an injection.
    \item[\upshape (2$'$)] The natural map $\Z^{(\simp\FF)} \to \Z^{(\supp\FF)}$ is an isomorphism, where we identify $\Z^{(\supp\FF)}$ with a subgroup of $\KKK_0(\mod\Lambda)$ generated by $[S]$ with $S \in \supp\FF$.
    \item $\#\simp\FF = \#\supp\FF$ holds.
  \end{enumerate}
  Moreover, the map in {\upshape (2$'$)} is always surjective.
\end{theorem}
\begin{proof}
  We give a proof using $\tau$-tilting theory and the Grothendieck group $\KKK_0(\FF)$ of the exact category $\FF$, for which we refer to \cite{AIR} and \cite{eno} respectively.
  It is shown in \cite[Theorem 4.12, Corollary 5.14]{eno} that the following are equivalent:
  \begin{enumerate}
    \item[(i)] $\FF$ satisfies (JHP).
    \item[(ii)] The natural map $\Z^{(\simp\FF)} \to \KKK_0(\FF)$, which is always surjective, is an isomorphism.
    \item[(iii)] $\# \simp\FF = | U |$ holds, where $U$ is a support $\tau^-$-tilting module with $\FF = \Sub U$ and $|U|$ is a number of non-isomorphic indecomposable direct summands of $U$.
  \end{enumerate}
  On the other hand, \cite[Proposition 2.2]{AIR} implies $|U| = \#\supp U$. Since $\FF = \Sub U$, clearly $\supp \FF = \supp U$ holds, hence we have $|U| = \#\supp \FF$.
  Therefore, (1) and (3) are equivalent.

  To see that (2) and (2$'$) are also equivalent, let us consider $\KKK_0(\FF)$. By using \cite[Lemma 5.7]{eno}, one can show that the natural map $\KKK_0(\FF) \to \KKK_0(\mod\Lambda)$ is an injection, and that its image is precisely $\Z^{(\supp\FF)}$. Thus all the conditions are equivalent.
\end{proof}

To describe a characterization of (JHP) for $\FF(w)$, we introduce the \emph{support of $w\in W$}. Recall that simple roots and simple reflections are parametrized by $Q_0$ in our setting.
\begin{definition}
  Let $w$ be an element of $W$. Then \emph{its support} is a subset $\supp(w)$ of $Q_0$ defined as follows:
  \[
  \supp(w) = \{ i \in Q_0 \, | \, \text{there is a reduced expression of $w$ which contains $s_i$} \}
  \]
\end{definition}
Then the support of $w$ coincides with the support of $\FF(w)$ in the following sense:
\begin{proposition}\label{binv:prop:suppsame}
  Let $w$ be an element in $W$. Then a natural bijection $Q_0 \xrightarrow{\sim} \simp(\mod\Pi)$, which sends $i$ to the simple $\Pi$-module corresponding to $i$, restricts to a bijection
  \[
  \supp (w) \xrightarrow{\sim} \supp \FF(w).
  \]
\end{proposition}
\begin{proof}
  For a positive root $\be \in \Phi^+$, we can write $\be = \sum_{i \in Q_0} n_i \al_i$ with $n_i\geq 0$ in a unique way.
  Denote by $\supp(\be)$ the set of $i \in Q_0$ with $n_i > 0$.
  Then Corollary \ref{binv:cor:diminv} clearly implies that the bijection $Q_0 \to \simp(\mod\Pi)$ restricts to a bijection
  \[
  \bigcup_{\be \in \inv(w)} \supp(\be) \xrightarrow{\sim} \supp \FF(w).
  \]
  Thus it suffices to show $\bigcup_{\be \in \inv(w)}\supp(\be) = \supp(w)$.

  Let $w = s_{u_1} \cdots s_{u_l}$ be a reduced expression of $w$, and $\be_1, \dots, \be_l$ its associated root sequence. Then we have $\inv(w) = \{\be_1,\dots,\be_l\}$.
  First suppose that $i$ belongs to $\bigcup_{\be \in \inv(w)}\supp(\be)$, then $i \in \supp(\be_m)$ for some $1 \leq m \leq l$. Recall that $\be_m = s_{u_1} \cdots s_{u_{m-1}}(\al_{u_m})$, and that $\{ \al_u \, | u \in Q_0 \}$ is a basis of $V$.
  Since each $s_u \colon V \to V$ changes only the $\al_u$-component of roots with respect to this basis, $i$ should appear in $\{u_1,u_2, \dots, u_m \}$. Thus $i \in \supp(w)$ holds.

  Conversely, suppose $i \in \supp(w)$. Take the minimal $m$ such that $u_m = i$ holds. We claim $i \in \supp(\be_m)$.
  Indeed, we have $\be_m = s_{u_1} s_{u_2} \cdots s_{u_{m-1}}(\al_i)$, and $i$ does not appear in $u_1,\dots,u_{m-1}$ by the minimality of $m$. Since $s_u$ changes only the $\al_u$-component, the $\al_i$-component of $\be_m$ is $1$, hence $i \in \supp(\be_m)$ holds.
\end{proof}

Now the following immediately follows from these observations.
\begin{theorem}\label{binv:thm:jhpmain}
  Let $w$ be an element of $W$. Then the following are equivalent:
  \begin{enumerate}
    \item $\FF(w)$ satisfies the Jordan-H\"older property.
    \item A map $\varphi_w \colon\Z^{\Binv(w)} \to \Z^{\supp(w)}$ defined by $\varphi_w(\sum_i n_i \al_i) =\sum_i n_i e_i $ for $\sum_i n_i \al_i \in \Binv(w)$ is a bijection, where $e_i$ denotes the basis of $\Z^{\supp(w)}$ corresponding to $i \in \supp(w)$.
    \item Elements in $\Binv(w)$ are linearly independent in $V$.
    \item $\#\Binv(w) = \#\supp(w)$ holds.
  \end{enumerate}
  Moreover, the map in {\upshape (2)} is always surjective.
\end{theorem}
\begin{proof}
  Since $\FF(w) = \Sub (\Pi/I_w)$ by definition, we can apply Theorem \ref{binv:thm:jhpchar} to $\FF(w)$. Hence the following are equivalent:
  \begin{enumerate}
    \item[(i)] $\FF(w)$ satisfies (JHP).
    \item[(ii)] The map $\Z^{(\simp\FF(w))} \to \Z^{\supp\FF(w)}$, which is always a surjection, is an isomorphism.
    \item[(iii)] $\#\simp\FF(w) = \#\supp \FF(w)$ holds.
  \end{enumerate}
  By identifying $Q_0$ with simple $\Pi$-modules, $\supp \FF(w)$ bijectively corresponds to $\supp(w)$ by Proposition \ref{binv:prop:suppsame}. Moreover, $\simp\FF(w)$ bijectively corresponds to $\Binv(w)$ by taking dimension vectors by Corollary \ref{binv:cor:main}.
  Thus the map in (ii) are exactly same as $\varphi_w$ in the assertion under the identification $\Z^{(\simp\FF(w))} \iso \Z^{\Binv(w)}$ and $\Z^{\supp\FF(w)} \iso \sum_{i \in \supp(w)}\Z \al_i$.

  The left hand side in (iii) is equal to $\#\Binv(w)$ by Corollary \ref{binv:cor:main}, and the right hand side is equal to $\#\supp(w)$ by Proposition \ref{binv:prop:suppsame}, thus (1), (2) and (4) are equivalent.
  Moreover, it is clear that (2) is equivalent to (3) since $\varphi_w$ is always surjective.
\end{proof}
We will use the map $\varphi$ above later in the appendix to relate our characterization of (JHP) to \emph{forest-like permutations} defined in \cite{bmb} and the Schubert variety $X_w$ for type A case.
\begin{remark}
  The equality $\#\Binv(w) = \#\supp(w)$ naturally arises when one consider the Bruhat interval in $W$ and its Poincar\'e polynomial.
  Let $w$ be an element of $W$ and $[e,w]$ the interval with respect to the Bruhat order.
  A \emph{Pincar\'e polynomial $P_w(q)$ of $w$} is defined by
  \[
  P_w(q) := \sum_{v \in [e,w]} q^{\ell(v)}.
  \]
  Let us write $P_w(q) = \sum_{i=0}^{\ell(w)} a_i q^i$. Then we have $\# \supp(w) = a_1$ and $\#\Binv(w) = a_{l-1}$, since supports of $w$ are precisely simple reflections which are below $w$ in the Bruhat order, and Bruhat inversions of $w$ are in bijection with elements which are covered by $w$ in the Bruhat order by Proposition \ref{binv:prop:bruhatcover}.
  Thus our criterion is equivalent to $a_1 = a_{l-1}$.
\end{remark}

\begin{example}
  Consider an element $w = s_{12312}$ in Example \ref{binv:ex:ainv}. Then there are three Bruhat inversions of $w$, namely, $100$, $010$ and $011$. This number is equal to the number of $\supp(w) = \{1,2,3\}$, thus $\FF(w)$ satisfies (JHP).

  On the other hand, consider an element $w=s_{012301230}$ in Example \ref{binv:ex:dinv}. Then we have $\supp(w) = \{0,1,2,3\}$, but a computation shows $\Binv(w) = \inv(w) \setminus \{\substack{1 \\ 121} \}$ (for example, this follows from the fact that deleting any letter from $s_{012301230}$ yields a reduced expression except for the middle $0$). so there are eight Bruhat inversions of $w$. Thus $\FF(w)$ does not satisfy (JHP).
\end{example}

\subsection{Conjectures}

In this subsection, we give some natural conjectures on the existence of the particular kind of short exact sequences related to Theorem \ref{binv:thm:main}.

The most non-trivial part of the proof of Theorem \ref{binv:thm:main} is to show that $B_m$ is non-simple if $\udim B_m$ is non-Bruhat. If $\udim B_m$ is non-Bruhat, then as in the proof, there is a brick sequence $B_1,\dots,B_i,\dots,B,\dots,B_j, \dots, B_l$ of $\FF(w)$ such that $\udim B_i + \udim B_j = \udim B$. Then the following conjecture naturally occurs.
\begin{conjecture}\label{binv:conj1}
  Let w be an element of $W$. Take a brick sequence $B_1,\dots, B_l$ of $\FF(w)$, and suppose $\udim B_m = \udim B_i + \udim B_j$ for $1 \leq i < m < j \leq l$. Then there is an exact sequence
  \[
  \begin{tikzcd}
    0 \rar & B_i \rar & B_m \rar & B_j \rar & 0.
  \end{tikzcd}
  \]
\end{conjecture}
In fact, in the proof of Theorem \ref{binv:thm:main}, we only construct a non-zero non-injection $\varphi \colon B_m \to B_j$, which is enough for our purpose.
This conjecture can be seen as a natural generalization of the result of Proposition \ref{binv:prop:dr}, where the path algebra case was shown over an algebraically closed field.

We have another conjecture on non-simple objects. A \emph{semibrick} $\SS$ in $\mod\Pi$ is a set of bricks in $\mod\Pi$ such that $\Hom_\Pi(S,T) = 0$ holds for every two distinct elements $S,T \in \SS$.
\begin{conjecture}\label{binv:conj2}
  Let w be an element of $W$ and $B$ a non-simple object in $\FF(w)$. Then there is a semibrick $\{S,T\}$ in $\FF(w)$ and an exact sequence
  \[
  \begin{tikzcd}
    0 \rar & S \rar & B \rar & T \rar & 0.
  \end{tikzcd}
  \]
\end{conjecture}
This conjecture is closely related to the lattice property (forcing order) of the interval $[e,w]$ or $[0,\FF(w)]$, and the root-theoretical combinatorial property of the inversion set (contractibility of inversion triples defined in \cite{gl}). The author has obtained the proof of Conjecture \ref{binv:conj2} for type $A_n,D_n$ using combinatorics of (signed) permutations and $E_6$ using computer.

Conjecture \ref{binv:conj2} can be shown to be equivalent to the following conjecture. Recall that a simple object in a torsion-free class $\FF$ appears in every brick sequence of $\FF$ by Proposition \ref{binv:prop:simpseq}. Then it is natural to ask whether the converse holds:
\begin{conjecture}
  Let $w$ be an element of $W$. If a brick $B$ appears in every brick sequence of $\FF(w)$, then $B$ is a simple object in $\FF(w)$.
\end{conjecture}
This conjecture makes sense for any torsion-free classes over any finite-dimensional algebras, but this fails in general. For example, consider $\GG_3$ in Example \ref{binv:ex:dirrt}. Then there is only one brick sequence of $\GG_3$, namely, $2, \substack{1\\2\\2},\substack{1\\2}$. However, $\substack{1\\2\\2}$ is non-simple in $\GG_3$.

\section{Torsion-free classes over path algebras of Dynkin type}\label{binv:sec:5}
In this section, we use the results in the previous section to study torsion-free classes over path algebras of Dynkin type.
\emph{Throughout this section, let $Q$ be a Dynkin quiver, and we use the same notation as in Section \ref{binv:sec:notation}}.
We have a natural surjection of algebras $\Pi \defl kQ$, defined by annihilating all arrows in $\ov{Q}$ which do not appear in $Q$. Thereby we have an embedding $\mod kQ \hookrightarrow \mod \Pi$, and we often identify $\mod kQ$ with a subcategory of $\mod\Pi$.

Let us recall the celebrated theorem of Gabriel:
\begin{theorem}\label{binv:thm:gabriel}
  The assignment $M \mapsto \udim M$ for $M \in \mod kQ$ induces a bijection
  \[
  \udim \colon \ind(\mod kQ) \xrightarrow{\sim} \Phi^+.
  \]
  In other words, indecomposable $kQ$-modules bijectively correspond to positive roots by taking dimension vectors.
\end{theorem}

\subsection{Coxeter-sortable elements and torsion-free classes}
We begin with introducing some terminology which we need to give a description of $\torf kQ$.
Put $n:= \# Q_0$. Then a \emph{Coxeter element $c_Q$ of $Q$} is an element $c_Q = s_{u_1} \cdots s_{u_n}$ of $W$ with $Q_0 = \{u_1,\dots,u_n\}$ which satisfies the following condition: if there is an arrow $i \ot j$ in $Q$, then $s_i$ appears before $s_j$ in this expression of $c_Q$.

Let $c = c_Q$ be a Coxeter element of $Q$, and $w$ an element of $W$.
We say that $w$ is \emph{$c$-sortable} if there exists a reduced expression of the form $w = c^{(0)} c^{(1)} \cdots c^{(m)}$ such that each $c^{(i)}$ is a subword of $c$ satisfying $\supp(c^{(0)}) \supset \supp(c^{(1)}) \supset \cdots \supset \supp(c^{(m)})$. We call such an expression a \emph{$c$-sorting word of $w$.}

Now we can state the classification of torsion-free classes in $\mod kQ$, which was first established by \cite{IT}, and then generalized to any acyclic quiver by \cite{AIRT} and \cite{thomas}.
\begin{theorem}[{\cite[Theorem 4.3]{IT}}]\label{binv:thm:itbij}
  Let $Q$ be a Dynkin quiver and $W$ its Weyl group. For $w \in W$, define a subcategory $\FF_Q(w)$ of $\mod kQ$ by
  \[
  \FF_Q(w) := \add \{ M \in \ind(\mod kQ) \, | \, \udim M \in \inv(w) \}.
  \]
  Then the assignment $w \mapsto \FF_Q(w)$ gives a bijection
  \[
  \{ w \, | \, \text{$w$ is $c_Q$-sortable} \} \xrightarrow{\sim} \torf kQ.
  \]
\end{theorem}

\subsection{Simple objects versus Bruhat inversions}
Let $w$ be a $c_Q$-sortable element, then we have a torsion-free class $\FF_Q(w)$ in $\mod kQ$ and a torsion-free class $\FF(w)$ in $\mod \Pi$. The relation between these two was stated implicitly in \cite{AIRT} and the proof was involved, thus we present a brief explanation of it.

We begin with the following observation.
\begin{proposition}[{\cite[Theorem 3.3]{AIRT}}]\label{binv:prop:inkq}
  Let $w$ be a $c_Q$-sortable element in $W$ and $B_1,\dots, B_l$ a brick sequence of $\FF(w) $ associated with a $c_Q$-sorting word of $w$. Then we have $B_i \in \mod kQ$ for each $i$.
\end{proposition}

Using this, we obtain the following description of $\FF_Q(w)$ via a brick sequence (c.f. \cite[Theorem 3.11]{AIRT}).
\begin{proposition}\label{binv:prop:fwadd}
  Let $w$ be a $c_Q$-sortable element in $W$, and let $B_1, \dots, B_l$ be a brick sequence of $\FF(w)$ associated with a $c_Q$-sorting word of $w$. Then we have $\FF_Q(w) = \add \{ B_1,\dots,B_l \}$. Moreover, $\FF_Q(w)  = \FF(w) \cap \mod kQ$ holds.
\end{proposition}
\begin{proof}
  By Proposition \ref{binv:prop:inkq}, we have $B_1, \dots, B_l \in \mod kQ$, and $\{ \udim B_1, \dots, \udim B_l \}$$ = \inv(w)$ holds by Proposition \ref{binv:prop:categorify}.
  Since Theorem \ref{binv:thm:gabriel} implies that there exists exactly one indecomposable $kQ$-module which has a fixed dimension vector, every indecomposable $kQ$-module $M$ with $\udim M \in \inv(w)$ should appear in $\{ B_1,\dots, B_l \}$.
  Therefore, by the definition of $\FF_Q(w)$, we have $\FF_Q(w) = \add \{ B_1, \dots, B_l \}$.

  We will prove $\FF_Q(w) = \FF(w) \cap \mod kQ$. Since each $B_i$ belongs to $\FF(w) \cap \mod kQ$, we have $\FF_Q(w) \subset \FF(w) \cap \mod kQ$. Conversely, let $M \in \FF(w) \cap \mod kQ$.
  Then we have $M \in \Filt(B_1, \dots, B_l)$ by Theorem \ref{binv:thm:brickseq}, where $\Filt$ is considered inside $\mod \Pi$.
  On the other hand, since $\mod kQ \subset \mod \Pi$ is closed under subquotients, clearly we have $M \in \Filt_{kQ}(B_1, \dots, B_l)$, where $\Filt_{kQ}$ means we consider it inside $\mod kQ$.
  However, $\add \{B_1,\dots,B_l\} = \FF_Q(w)$ is known to be closed under extensions in $\mod kQ$ since it is a torsion-free class by Theorem \ref{binv:thm:itbij}. Thus $M \in \FF_Q(w)$ holds.
\end{proof}

Now we can state our classification of simple objects in $\FF_Q(w)$.
\begin{theorem}\label{binv:thm:pathmain}
  Let $w$ be a $c_Q$-sortable element of $W$. Then a bijection $\udim \colon \ind \FF_Q(w) \to \inv(w)$ restricts to a bijection
  \[
  \udim \colon \simp \FF_Q(w) \xrightarrow{\sim} \Binv(w).
  \]
  In other words, simple objects in $\FF_Q(w)$ bijectively correspond to Bruhat inversions of $w$ by taking dimension vectors.
\end{theorem}
\begin{proof}
  Let $B_1, \dots, B_l$ be a brick sequence of $\FF(w)$ (not $\FF_Q(w)$!) associated to a $c_Q$-sorting word of $w$. Then Proposition \ref{binv:prop:fwadd} says that $\FF_Q(w) = \{ B_1, \dots, B_l \}$.
  Thus it suffices to show the following:

  {\bf (Claim)}: \emph{The following are equivalent for $1 \leq m \leq l$:
  \begin{enumerate}
    \item $\udim B_m \in \Binv(w)$ holds.
    \item $B_m$ is a simple object in $\FF(w)$
    \item $B_m$ is a simple object in $\FF_Q(w)$.
  \end{enumerate}
  }
  The equivalence of (1) and (2) is nothing but Theorem \ref{binv:thm:main}, thus it suffices to show that (2) and (3) are equivalent. This immediately follows from the fact that $\FF_Q(w) = \FF(w) \cap \mod kQ$ holds by Proposition \ref{binv:prop:fwadd} and that $\mod kQ$ is closed under subquotients in $\mod\Pi$.
\end{proof}

As in the case of $\FF(w)$, we can characterize the validity of the Jordan-H\"older property as follows.
\begin{corollary}\label{binv:cor:jhppath}
  Let $w$ be an element of $W$. Then $\FF_Q(w)$ satisfies (JHP) if and only if $\#\Binv(w) = \# \supp(w)$ holds.
\end{corollary}
\begin{proof}
  Immediate from Theorems \ref{binv:thm:jhpchar} and \ref{binv:thm:pathmain}, once we observe that $\supp\FF(w)$ are in bijection with $\supp(w)$, which can be proved similarly to Proposition \ref{binv:prop:suppsame}.
\end{proof}
These results generalize the results in \cite{eno}, in which the type A case was proved by direct computation of representations of $Q$.

If we assume that the base field $k$ is algebraically closed, then we can use the following result of \cite{dr} to give a more quick proof of Theorem \ref{binv:thm:pathmain} without using preprojective algebras.
\begin{proposition}[\cite{dr}]\label{binv:prop:dr}
  Let $k$ be an algebraically closed field and $Q$ a Dynkin quiver. Take indecomposable $kQ$-modules $L,M,N$ such that $\udim L + \udim N = \udim M$ holds in $\Phi^+$. Then by interchanging $L$ and $N$ if necessary, there is an exact sequence in $\mod kQ$ of the following form:
  \[
  \begin{tikzcd}
    0 \rar & L \rar & M \rar & N \rar & 0
  \end{tikzcd}
  \]
\end{proposition}
Note that the proof of this given in \cite{dr} is algebro-geometric. The author do not know whether the same method can be used to study preprojective algebras, and whether this kind of exact sequence always exists in the preprojective algebra case (see Conjecture \ref{binv:conj1}).

Now we can give another proof of Theorem \ref{binv:thm:pathmain} provided that $k$ is algebraically closed.
\begin{proof}[Proof of Theorem \ref{binv:thm:pathmain}]
  Recall that we have a bijection $\udim \colon \ind\FF_Q(w) \to \inv(w)$.
  We will show that $M\in \ind\FF_Q(w)$ is simple in $\FF_Q(w)$ if and only if $\udim M \in \Binv(w)$.
  If $M$ is not simple, then Theorem \ref{binv:thm:binvchar} clearly implies that $\udim M$ is a non-Bruhat inversion of $w$ by taking dimension vectors.
  Thus it suffices to show that if $\udim M$ is a non-Bruhat inversion of $w$, then $M$ is not a simple object in $\FF_Q(w)$.

  By Theorem \ref{binv:thm:binvchar}, there are $\al,\be \in \inv(w)$ such that $\udim M = \al + \be$ holds since $\Phi$ is simply-laced.
  Take indecomposable $kQ$-modules $M_\al$ and $M_\be$ with $\udim M_\al = \al$ and $\udim M_\be = \be$, which exist by Theorem \ref{binv:thm:gabriel}. By definition we have $M_\al,M_\be \in \FF_Q(w)$ holds.
  Then Proposition \ref{binv:prop:dr} implies that there exists a short exact sequence
  \[
  \begin{tikzcd}
    0 \rar & M_\al \rar & M \rar & M_\be \rar & 0
  \end{tikzcd}
  \]
  by interchanging $\al$ and $\be$ if necessary. Clearly this implies that $M$ is not a simple object in $\FF_Q(w)$.
\end{proof}

\begin{appendix}

\section{Description and enumeration for type A}
In this appendix, we focus on type A case and give more explicit and combinatorial description of results in this paper. First, we give an explicit diagrammatic construction of simple objects in $\FF(w)$ by using a \emph{Bruhat inversion graph $G_w$}. Next, we characterize elements $w$ such that $\FF(w)$ satisfies (JHP) in terms of $G_w$, and deduce some numerical consequences.

\emph{Throughout this appendix, we will use the following notation}:
\begin{itemize}
  \item $Q$ is a quiver whose underlying graph is the following Dynkin diagram of type $A_n$:
  \[
  \begin{tikzcd}
    1 \rar[dash] & 2 \rar[dash] & \cdots \rar[dash] & n.
  \end{tikzcd}
  \]
  \item $\Phi$ is a standard root system of type $A_n$ in $V:= \R^{n+1}$, that is, $\Phi = \{ \vare_i - \vare_j \, | \, 1 \leq i,j \leq n+1 \}$, where $\vare_i$ denotes the standard basis of $V$.
  \item We fix simple roots by $\al_i := \vare_i - \vare_{i+1}$ for $1 \leq i \leq n$.
  \item $W$ is the Weyl group of $\Phi$, and we often identify $W = S_{n+1}$ with the symmetric group $S_{n+1}$ acting on the set $[n+1]:=\{1,2,\dots,n,n+1\}$ so that $w(\vare_i) = \vare_{w(i)}$ holds.
  \item For $i,j \in [n+1]$, we denote by $(i \,\, j) \in S_{n+1}$ the transposition of the letter $i$ and $j$, and put $\be_{(i,j)}:= \vare_i - \vare_j$. Then $(i\,\,j)$ is identified with the reflection with respect to $\be_{(i,j)}$.
  \item For $w \in S_{n+1}$, we often use the \emph{one-line notation for $w$}, that is, we write as $w = w(1)w(2) \cdots w(n+1)$.
  \item $\Pi$ is the preprojective algebra of $\Phi$.
  \item $\FF(w) \in \torf\Pi$ is the torsion-free class in $\mod \Pi$ defined in Theorem \ref{binv:thm:mizuno}.
\end{itemize}

First, let us introduce the combinatorial variants of (Bruhat) inversion sets.
\begin{definition}
  Let $w$ be an element of $S_{n+1}$.
  \begin{enumerate}
    \item $\Inv(w)$ consists of a pair $(i,j)$ with $1 \leq i < j \leq n+1$ such that the letter $j$ appears left to $i$ in the one-line notation for $w$.
    \item $\BInv(w)$ consists of a pair $(i,j)$ with $1 \leq i < j \leq n+1$ such that the letter $j$ appears left to $i$ and there is no $k$ with $i < k < j$ such that the letter $k$ appears between $j$ and $i$.
  \end{enumerate}
\end{definition}
This notation is justified by the following, which can be proved by direct calculation.
\begin{proposition}
  Let $w$ be an element of $W = S_{n+1}$ and $i,j \in [n+1]$. Then the following hold.
  \begin{enumerate}
    \item $\be_{(i,j)} \in \inv(w)$ if and only if $(i,j) \in \Inv(w)$.
    \item $\be_{(i,j)} \in \Binv(w)$ if and only if $(i,j) \in \BInv(w)$.
  \end{enumerate}
\end{proposition}
Let us introduce a way to visualize Bruhat inversions, a \emph{Bruhat inversion graph}. Let $w$ be an element of $W = S_{n+1}$. Consider a square array of boxes with $(n+1)$ rows and $(n+1)$ columns. We name $(i,j)$ to the box in the $i$-th row and $j$-th column, and put a dot in $(i,w(i))$ for each $1 \leq i \leq n+1$. We call it a \emph{diagram of $w$}.
A \emph{Bruhat inversion graph $G_w$} is obtained by connecting every two dots in the diagram of $w$ which correspond to the Bruhat inversion of $w$, that is, we connect $(w^{-1}(i),i)$ and $(w^{-1}(j),j)$ if $(i,j) \in \BInv(w)$ holds.
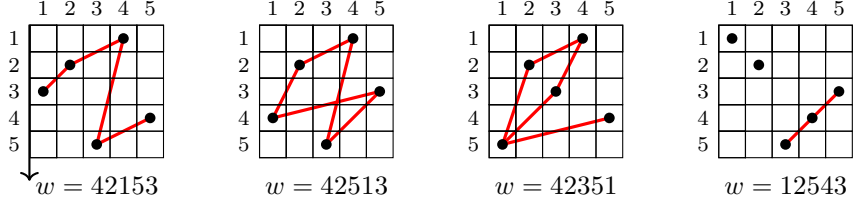
\begin{figure}[h]
  \centering
  \begin{tikzpicture}[
      mycell/.style={draw, minimum size=1em},
      dot/.style={mycell,
          append after command={\pgfextra \fill (\tikzlastnode) circle[radius=.2em]; \endpgfextra}}]

  \matrix (m) [matrix of nodes, row sep=-\pgflinewidth, column sep=-\pgflinewidth,
      nodes={mycell}, nodes in empty cells]
  {
  &&&|[dot]|&\\
  &|[dot]|&&&\\
  |[dot]|&&&&\\
  &&&&|[dot]|\\
  &&|[dot]|&&\\
  };

  \begin{scope}[on background layer, every path/.style={red, very thick}]
    \draw (m-1-4.center) -- (m-2-2.center);
    \draw (m-2-2.center) -- (m-3-1.center);
    \draw (m-1-4.center) -- (m-5-3.center);
    \draw (m-4-5.center) -- (m-5-3.center);
  \end{scope}
  \foreach \i [count=\xi from 1] in  {1,...,5}{
      \node[mycell, label=above:\footnotesize\xi] at (m-1-\i) {};
      \node[mycell,label=left:\footnotesize\xi] at (m-\i-1) {};
  }
  \node[below=3mm of m-5-1.south west] (A) {};
  \draw[->,thick] (m-1-1.north west) -- (A);

  \matrix (n) [matrix of nodes, row sep=-\pgflinewidth, column sep=-\pgflinewidth,
      nodes={mycell}, nodes in empty cells, right=of m]
  {
  &&&|[dot]|&\\
  &|[dot]|&&&\\
  &&&&|[dot]|\\
  |[dot]|&&&&\\
  &&|[dot]|&&\\
  };

  \begin{scope}[on background layer, every path/.style={red, very thick}]
    \draw (n-1-4.center) -- (n-2-2.center);
    \draw (n-2-2.center) -- (n-4-1.center);
    \draw (n-1-4.center) -- (n-5-3.center);
    \draw (n-3-5.center) -- (n-5-3.center);
    \draw (n-3-5.center) -- (n-4-1.center);
  \end{scope}
  \foreach \i [count=\xi from 1] in  {1,...,5}{
      \node[mycell, label=above:\footnotesize\xi] at (n-1-\i) {};
      \node[mycell,label=left:\footnotesize\xi] at (n-\i-1) {};
  }

  \matrix (nn) [matrix of nodes, row sep=-\pgflinewidth, column sep=-\pgflinewidth,
      nodes={mycell}, nodes in empty cells, right=of n]
  {
  &&&|[dot]|&\\
  &|[dot]|&&&\\
  &&|[dot]|&&\\
  &&&&|[dot]|\\
  |[dot]|&&&&\\
  };

  \begin{scope}[on background layer, every path/.style={red, very thick}]
    \draw (nn-1-4.center) -- (nn-2-2.center);
    \draw (nn-1-4.center) -- (nn-3-3.center);
    \draw (nn-2-2.center) -- (nn-5-1.center);
    \draw (nn-3-3.center) -- (nn-5-1.center);
    \draw (nn-4-5.center) -- (nn-5-1.center);
  \end{scope}
  \foreach \i [count=\xi from 1] in  {1,...,5}{
      \node[mycell, label=above:\footnotesize\xi] at (nn-1-\i) {};
      \node[mycell,label=left:\footnotesize\xi] at (nn-\i-1) {};
  }

  \matrix (mm) [matrix of nodes, row sep=-\pgflinewidth, column sep=-\pgflinewidth,
      nodes={mycell}, nodes in empty cells, right=of nn]
  {
  |[dot]|&&&&\\
  &|[dot]|&&&\\
  &&&&|[dot]|\\
  &&&|[dot]|&\\
  &&|[dot]|&&\\
  };

  \begin{scope}[on background layer, every path/.style={red, very thick}]
    \draw (mm-3-5.center) -- (mm-4-4.center);
    \draw (mm-4-4.center) -- (mm-5-3.center);
  \end{scope}
  \foreach \i [count=\xi from 1] in  {1,...,5}{
      \node[mycell, label=above:\footnotesize\xi] at (mm-1-\i) {};
      \node[mycell,label=left:\footnotesize\xi] at (mm-\i-1) {};
  }

  \node[below=0mm of m] {$w=42153$};
  \node[below=0mm of n] {$w=42513$};
  \node[below=0mm of nn] {$w=42351$};
  \node[below=0mm of mm] {$w=12543$};

  \end{tikzpicture}
  \caption{Bruhat inversion graphs $G_w$}
  \label{binv:fig:graph}
\end{figure}
Figure \ref{binv:fig:graph} is examples for $G_w$ for elements $w$ in $S_5$. It is clear from definition that $G_w$ is obtained by connecting every two dots in the diagram of $w$ which look as follows,
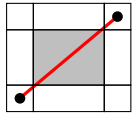
\[
\begin{tikzpicture}[
    mycell/.style={draw, minimum size=1em},
    dot/.style={mycell,
        append after command={\pgfextra \fill (\tikzlastnode) circle[radius=.2em]; \endpgfextra}}]
  \matrix (m) [matrix of nodes, row sep=-\pgflinewidth, column sep=-\pgflinewidth,
      nodes={mycell}, nodes in empty cells, row 2/.style={text height=5mm}, column 2/.style={text width = 7mm}]
  {
  & &|[dot]|\\
  & & \\
  |[dot]|&  &\\
  };

  \begin{scope}[on background layer, every path/.style={red, very thick}]
    \fill[gray!50] (m-2-2.north east) rectangle (m-2-2.south west);
    \draw (m-1-3.center) -- (m-3-1.center);

  \end{scope}
\end{tikzpicture}
\]
such that there are no dots in the gray region. It is also clear from definition that an edge in $G_w$ bijectively corresponds to a Bruhat inversion of $w$.

\subsection{Diagrammatic description of simples in $\FF(w)$}\label{binv:sec:A1}
In this subsection, we will give a combinatorial description of simple objects in $\FF(w)$ using \emph{arc diagrams} introduced in \cite{readingarc} and the description of bricks given in \cite{asai2}.
The author would like to thank Y. Mizuno for explaining to him the interpretation of the description in \cite{asai2} in terms of arc diagrams.

Let $w$ be an element of $S_{n+1} = W$. We will construct a $\Pi$-module $B_e$ for each edge $e$ in $G_\pi$ in the following way:
\begin{enumerate}
  \item Remove all the edges in $G_w$ except $e$.
  \item Move down all the dots into a single horizontal line, allowing $e$ to curve, but not to pass through any dots. We call this diagram an \emph{arc diagram} of $e$.
  \item Draw $n$ vertical dashed lines between adjacent dots in the arc diagram, and name these lines as $1,2,\dots,n$ from left to right.
  \item Define a (not necessarily full) subquiver $Q(e)$ of $\ov{Q}$ by the following rule:
  \begin{itemize}
    \item The vertex set of $Q(e)$ consists of $i \in \ov{Q}_0$ such that $e$ and the line $i$ intersect in the arc diagram of $e$.
    \item Suppose that we have $i, i+1 \in Q(e)_0$.
    If the segment of $e$ cut by the lines $i$ and $i+1$ is above the unique dot between these lines, then we put an arrow $i \to i+1$, and put $i \ot i+1$ if the segment is below the dot.
  \end{itemize}
  We call $Q(e)$ \emph{the defining quiver of $B_e$.}
  \item Define a $\Pi$-module $B_e$ as follows, where we construct $B_e$ as a representation of $\ov{Q}$.
  \begin{itemize}
    \item To each $i \in \ov{Q}_0$, we assign $k$ if $i \in Q(e)_0$, and $0$ otherwise.
    \item To each arrow $i \to j \in \ov{Q}$, we assign the identity map if $i \to j \in Q(e)_1$, and $0$ otherwise.
  \end{itemize}
  Since two cycles in $k\ov{Q}$ annihilates $B_e$ by construction, we can regard $B_e$ as a $\Pi$-module.
\end{enumerate}

\begin{example}
  Figure \ref{binv:fig:bex1} is an example of this construction for $w=42513$ and all the edges in $G_w$. The middle part is an arc diagram of three edges, and the right part shows defining quivers of $B_e$ corresponding to magenta, green, red, gray and blue edges from top to bottom.
  \begin{figure}[h]
    \centering
    \caption{Example of $B_e$ for $w=42513$}
    \label{binv:fig:bex1}
    \begin{tikzpicture}
      [ mycell/.style={draw, minimum size=1em},
        dot/.style={mycell,
            append after command={\pgfextra \fill (\tikzlastnode) circle[radius=.2em]; \endpgfextra}}]

      \matrix (n) [matrix of nodes, row sep=-\pgflinewidth, column sep=-\pgflinewidth,
          nodes={mycell}, nodes in empty cells]
      {
      &&&|[dot]|&\\
      &|[dot]|&&&\\
      &&&&|[dot]|\\
      |[dot]|&&&&\\
      &&|[dot]|&&\\
      };

      \node[below=0mm of n] {$G_w$};

      \begin{scope}[on background layer, every path/.style={very thick}]
        \draw[red] (n-3-5.center) -- (n-4-1.center);
        \draw[blue] (n-3-5.center) -- (n-5-3.center);
        \draw[gray] (n-1-4.center) -- (n-5-3.center);
        \draw[green] (n-1-4.center) -- (n-2-2.center);
        \draw[magenta] (n-2-2.center) -- (n-4-1.center);
      \end{scope}

      \begin{scope}[every node/.style={circle, fill=black, inner sep=.5mm, outer sep=0}]
        \node[right=of n] (1) {};
        \node[right=5mm of 1] (2) {};
        \node[right=5mm of 2] (3) {};
        \node[right=5mm of 3] (4) {};
        \node[right=5mm of 4] (5) {};
      \end{scope}

      \begin{scope}[every path/.style={dashed}, every node/.style={font=\footnotesize}]
        \draw ($(1)!0.5!(2) + (0,7mm)$) node[above] {$1$} -- +(0,-14mm);
        \draw ($(2)!0.5!(3) + (0,7mm)$) node[above] {$2$} -- +(0,-14mm);
        \draw ($(3)!0.5!(4) + (0,7mm)$) node[above] {$3$} -- +(0,-14mm);
        \draw ($(4)!0.5!(5) + (0,7mm)$) node[above] {$4$} -- +(0,-14mm);
      \end{scope}

      \begin{scope}
        [thick, rounded corners=8pt]
        \draw[red]
        (1) -- ($(2) - (0,5mm)$) -- ($(3) + (0,5mm)$) -- ($(4) - (0,5mm)$) -- (5) ;
        \draw[blue]
        (3) -- ($(4) - (0,7mm)$) -- (5);
        \draw[gray]
        (3) -- (4);
        \draw[green]
        (2) -- ($(3) + (0,7mm)$) -- (4);
        \draw[magenta]
        (1) -- (2);
      \end{scope}

      \node[below= of 3] {Arc diagrams};

      \matrix (c)
      [matrix of math nodes, column sep=5mm, right=15mm of 5, nodes in empty cells]
      {
      1 & 2 & 3 & 4 \\
      1 & & & & & \\
      & 2 & 3 & & & \\
      1 & 2 & 3 & 4 & &\\
      &  & 3 &&& \\
      \phantom{1}&  & 3 & 4 & &\\
      };

      \draw[decorate, decoration={brace, mirror}, thick] (c-2-1.west) -- (c-6-1.west);

      \node[left=1mm of c-1-1] {$\ov{Q}$:};

      \draw[<->] (c-1-1) -- (c-1-2);
      \draw[<->] (c-1-2) -- (c-1-3);
      \draw[<->] (c-1-3) -- (c-1-4);

      \begin{scope}[every path/.style={->,thick}]
        \draw[green] (c-3-2) -- (c-3-3);
        \draw[red] (c-4-2) -- (c-4-1);
        \draw[red] (c-4-2) -- (c-4-3);
        \draw[red] (c-4-4) -- (c-4-3);
        \draw[blue] (c-6-4) -- (c-6-3);
      \end{scope}
      \draw[magenta,thick] (c-2-5) -- (c-2-6);
      \draw[green,thick] (c-3-5) -- (c-3-6);
      \draw[red,thick] (c-4-5) -- (c-4-6);
      \draw[gray,thick] (c-5-5) -- (c-5-6);
      \draw[blue,thick] (c-6-5) -- (c-6-6);
      \node[left=1mm of c-4-1] {$Q(e)$:};
    \end{tikzpicture}
  \end{figure}

  Figure \ref{binv:fig:bex2} is an example for $w=42351$ and two particular edges in $G_w$. Note that the orientations of edges between $2$ and $3$ in $Q(e)$ may differ as in this example.
  
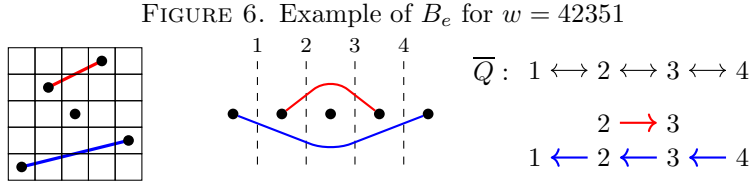
\begin{figure}[h]
    \centering
    \caption{Example of $B_e$ for $w=42351$}
    \label{binv:fig:bex2}
    \begin{tikzpicture}
      [ mycell/.style={draw, minimum size=1em},
        dot/.style={mycell,
            append after command={\pgfextra \fill (\tikzlastnode) circle[radius=.2em]; \endpgfextra}}]

      \matrix (n) [matrix of nodes, row sep=-\pgflinewidth, column sep=-\pgflinewidth,
          nodes={mycell}, nodes in empty cells]
      {
      &&&|[dot]|&\\
      &|[dot]|&&&\\
      &&|[dot]|&&\\
      &&&&|[dot]|\\
      |[dot]|&&&&\\
      };

      \begin{scope}[on background layer, every path/.style={very thick}]
        \draw[red] (n-1-4.center) -- (n-2-2.center);
        \draw[blue] (n-4-5.center) -- (n-5-1.center);
      \end{scope}

      \begin{scope}[every node/.style={circle, fill=black, inner sep=.5mm, outer sep=0}]
        \node[right=of n] (1) {};
        \node[right=5mm of 1] (2) {};
        \node[right=5mm of 2] (3) {};
        \node[right=5mm of 3] (4) {};
        \node[right=5mm of 4] (5) {};
      \end{scope}

      \begin{scope}[every path/.style={dashed}, every node/.style={font=\footnotesize}]
        \draw ($(1)!0.5!(2) + (0,7mm)$) node[above] {$1$} -- +(0,-14mm);
        \draw ($(2)!0.5!(3) + (0,7mm)$) node[above] {$2$} -- +(0,-14mm);
        \draw ($(3)!0.5!(4) + (0,7mm)$) node[above] {$3$} -- +(0,-14mm);
        \draw ($(4)!0.5!(5) + (0,7mm)$) node[above] {$4$} -- +(0,-14mm);
      \end{scope}

      \begin{scope}
        [thick, rounded corners=8pt]
        \draw[blue]
        (1) -- ($(3) - (0,5mm)$) -- (5) ;
        \draw[red]
        (2) -- ($(3) + (0,5mm)$) -- (4);
      \end{scope}

      \matrix (c)
      [matrix of math nodes, column sep=5mm, right=of 5, nodes in empty cells]
      {
      1 & 2 & 3 & 4 \\
      & & & \\
       & 2 & 3 &  \\
      1 & 2 & 3 & 4 \\
      };

      \node[left=0mm of c-1-1] {$\ov{Q}:$};

      \draw[<->] (c-1-1) -- (c-1-2);
      \draw[<->] (c-1-2) -- (c-1-3);
      \draw[<->] (c-1-3) -- (c-1-4);

      \begin{scope}[every path/.style={->,thick}]
        \draw[red] (c-3-2) -- (c-3-3);
        \draw[blue] (c-4-2) -- (c-4-1);
        \draw[blue] (c-4-3) -- (c-4-2);
        \draw[blue] (c-4-4) -- (c-4-3);
      \end{scope}
    \end{tikzpicture}
  \end{figure}
\end{example}

The above construction of arc diagrams is due to \cite{readingarc}. More precisely, in \cite{readingarc}, arcs were assigned only to \emph{descents} of $w$, which are inversions $(i,j) \in \Inv(w)$ such that $w$ is of the form $\cdots j i \cdots$.
Similarly, our construction of $\Pi$-modules $B_e$ is a generalization of the one given in \cite[Theorem 4.6]{asai2}, where $B_e$ was given (without using arc diagrams) for elements $w$ with unique descent.

We will confirm that $B_e$ is the simple object in $\FF(w)$ associated with the Bruhat inversion corresponding to $e$.
\begin{proposition}\label{binv:prop:simpdesc}
  Let $w$ be an element of $W = S_{n+1}$. Take $(i,j) \in \BInv(w)$, and let $e$ denotes the edge in $G_w$ corresponding to it. Then $B_e$ is the unique simple object in $\FF(w)$ with $\udim B_e = \be_{(i,j)}$.
\end{proposition}
\begin{proof}
  First, we will construct another element $w_e$ of $W$ with a unique descent.
  The following picture illustrates the construction, where all the dots lie in the gray regions.
  
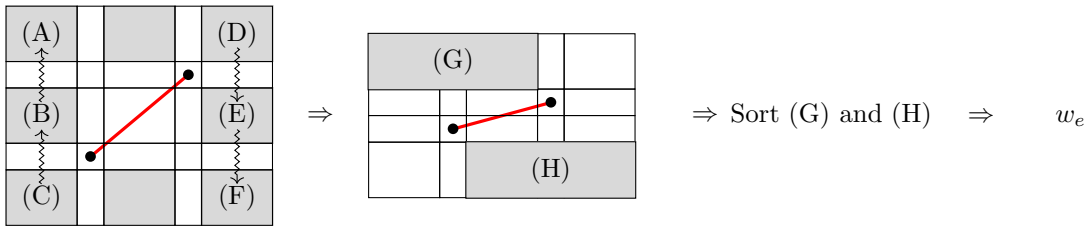
\[
  \begin{tikzpicture}[
      mycell/.style={draw, minimum size=1em},
      dot/.style={mycell,
          append after command={\pgfextra \fill (\tikzlastnode) circle[radius=.2em]; \endpgfextra}}]
    \matrix (m) [matrix of nodes, row sep=-\pgflinewidth, column sep=-\pgflinewidth,
        nodes={mycell}, nodes in empty cells, every odd row/.style={text height=5mm}, every odd column/.style={text width = 7mm}]
    {
    & & & &\\
    && &|[dot]|&\\
    && & &\\
    &|[dot]|&  &&\\
    & & & &\\
    };
    \begin{scope}[every node/.style={inner sep=0mm}]
      \node at (m-1-1.center) (A) {(A)};
      \node at (m-3-1.center) (B) {(B)};
      \node at (m-5-1.center) (C) {(C)};
      \node at (m-1-5.center) (D) {(D)};
      \node at (m-3-5.center) (E) {(E)};
      \node at (m-5-5.center) (F) {(F)};
    \end{scope}

    \begin{scope}[every path/.style={->, decorate, decoration={zigzag, amplitude=.3mm, segment length=1mm, post length=1mm}}]
      \draw (B) -- (A);
      \draw (C) -- (B);
      \draw (D) -- (E);
      \draw (E) -- (F);
    \end{scope}

    \begin{scope}
      [on background layer, every path/.style={red, very thick, fill={gray!30}}]
      \fill (m-1-1.north east) rectangle (m-1-1.south west);
      \fill (m-1-3.north east) rectangle (m-1-3.south west);
      \fill (m-1-5.north east) rectangle (m-1-5.south west);
      \fill (m-3-1.north east) rectangle (m-3-1.south west);
      \fill (m-3-5.north east) rectangle (m-3-5.south west);
      \fill (m-5-1.north east) rectangle (m-5-1.south west);
      \fill (m-5-3.north east) rectangle (m-5-3.south west);
      \fill (m-5-5.north east) rectangle (m-5-5.south west);
      \draw (m-2-4.center) -- (m-4-2.center);
    \end{scope}

    \matrix (n) [right=of m, matrix of nodes, row sep=-\pgflinewidth, column sep=-\pgflinewidth,
        nodes={mycell}, nodes in empty cells, row 1/.style={text height=5mm}, row 4/.style={text height=5mm}, every odd column/.style={text width = 7mm}]
    {
    & & & &\\
    && &|[dot]|&\\
    &|[dot]|&  &&\\
    & & & &\\
    };

    \begin{scope}
      [on background layer, every path/.style={red, very thick}]
      \draw (n-2-4.center) -- (n-3-2.center);
    \end{scope}

    \draw[fill=gray!30] (n-1-1.north west) rectangle (n-1-3.south east);
    \draw[fill=gray!30] (n-4-3.north west) rectangle (n-4-5.south east);

    \node at (n-1-2.center) {(G)};
    \node at (n-4-4.center) {(H)};

    \node at ($(m)!0.5!(n)$) {$\Rightarrow$};

    \node[right=of n] (s) {Sort (G) and (H) $\quad\Rightarrow\qquad w_e$};

    \node at ($(n)!0.5!(s)$) {$\Rightarrow$};
  \end{tikzpicture}
  \]
  The leftmost diagram is the diagram of $w$, and the red edge indicates $e$. Then perform the following procedure, requiring that all the diagrams in each step are diagrams of some elements in $W$:
  \begin{enumerate}
    \item Move all the dots in (B) and (C) to (A), and those in (D) and (E) to (F).
    \item Sort all the dots in (G) and (H) so that the column number increases from top to bottom.
  \end{enumerate}
  Denote by $w_e$ the resulting element. Then it is clear from the construction that $w_e$ has the unique descent. See Figure \ref{binv:fig:proof} for the example of this process, where $w=56723814$.

  Alternatively, in terms of the one-line notation, we can describe $w_e$ as follow. Let $w = \cdots (a) \cdots j \cdots (b) \cdots i \cdots (c) \cdots$ be the one-line notation for $w$.
  First, move all the numbers in $(b)$ and $(c)$ which are smaller than $i$ to $(a)$, and all the numbers in $(a)$ and $(b)$ which are larger than $j$ to $(c)$.
  Then we obtain an element of the form $\cdots (a) \cdots j i \cdots (c) \cdots$, since $(i,j)$ is a Bruhat inversion of $w$.
  Next, sort the part $(a)$ and $(c)$ in ascending order, and denote by $w_e$ the resulting element. Then $w_e = \cdots \un{j i} \cdots$ has the unique descent at the underlined part.

  It is straightforward to see that the module $B_e$ is the same as the module $B_{w_e}:=B_{e'}$, where $e'$ is the (unique) edge in the Bruhat inversion graph of $w_e$.
  Moreover, it is easily checked that $B_{w_e}$ is nothing but the brick constructed in \cite[Theorem 4.6]{asai2} associated to $w_e$, which has the unique descent.

  It is shown in \cite[Theorem 3.1]{asai2} that $B_{w_e}$ is the label of the unique arrow starting at $\FF(w_e)$. In particular, we have $B_{w_e} \in \brick\FF(w_e)$.
  On the other hand, by construction, it is straightforward to check that $\Inv(w_e) \subset \Inv(w)$ holds, which implies $w_e \leq w$ in the right weak order in $W$, see e.g. \cite[Proposition 3.1.3]{bb}. Thus we have $\FF(w_e) \subset \FF(w)$ in $\torf\Pi$ by Theorem \ref{binv:thm:mizuno}. Hence we have $B_{w_e} \in \brick\FF(w)$.
  Therefore, Corollary \ref{binv:cor:main} implies that $B_{w_e}$ is the unique simple object in $\FF(w)$ with its dimension vector $\be_{(i,j)}$, since $\be_{(i,j)} \in \Binv(w)$.
\end{proof}
By this, we can obtain all the simple objects in $\FF(w)$ by computing $B_e$ for each edge $e$ in $G_w$, as we have done in Figure \ref{binv:fig:bex1}.

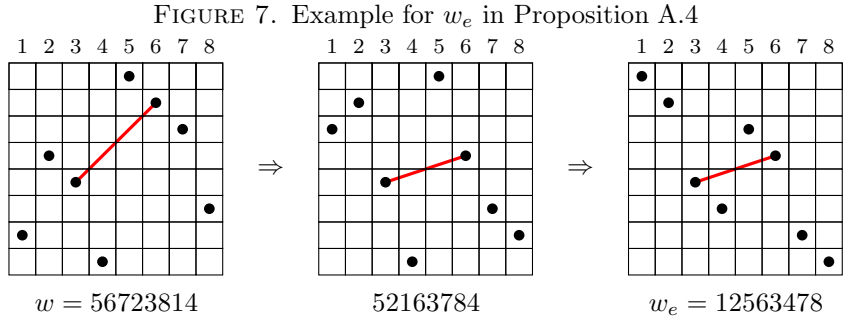
\begin{figure}[h]
  \centering
  \caption{Example for $w_e$ in Proposition \ref{binv:prop:simpdesc}}
  \label{binv:fig:proof}
  \begin{tikzpicture}
    [ mycell/.style={draw, minimum size=1em},
      dot/.style={mycell,
          append after command={\pgfextra \fill (\tikzlastnode) circle[radius=.2em]; \endpgfextra}}]

    \matrix (n) [matrix of nodes, row sep=-\pgflinewidth, column sep=-\pgflinewidth,
        nodes={mycell}, nodes in empty cells]
    {
    &&&&|[dot]|&&&\\
    &&&&&|[dot]|&&\\
    &&&&&&|[dot]|&\\
    &|[dot]|&&&&&&\\
    &&|[dot]|&&&&&\\
    &&&&&&&|[dot]|\\
    |[dot]|&&&&&&&\\
    &&&|[dot]|&&&&\\
    };

    \begin{scope}[on background layer, every path/.style={very thick}]
      \draw[red] (n-2-6.center) -- (n-5-3.center);
    \end{scope}

    \matrix (m) [right=of n, matrix of nodes, row sep=-\pgflinewidth, column sep=-\pgflinewidth,
        nodes={mycell}, nodes in empty cells]
    {
    &&&&|[dot]|&&&\\
    &|[dot]|&&&&&&\\
    |[dot]|&&&&&&&\\
    &&&&&|[dot]|&&\\
    &&|[dot]|&&&&&\\
    &&&&&&|[dot]|&\\
    &&&&&&&|[dot]|\\
    &&&|[dot]|&&&&\\
    };

    \begin{scope}[on background layer, every path/.style={very thick}]
      \draw[red] (m-4-6.center) -- (m-5-3.center);
    \end{scope}

    \matrix (mm) [right=of m, matrix of nodes, row sep=-\pgflinewidth, column sep=-\pgflinewidth,
        nodes={mycell}, nodes in empty cells]
    {
    |[dot]|&&&&&&&\\
    &|[dot]|&&&&&&\\
    &&&&|[dot]|&&&\\
    &&&&&|[dot]|&&\\
    &&|[dot]|&&&&&\\
    &&&|[dot]|&&&&\\
    &&&&&&|[dot]|&\\
    &&&&&&&|[dot]|\\
    };

    \begin{scope}[on background layer, every path/.style={very thick}]
      \draw[red] (mm-4-6.center) -- (mm-5-3.center);
    \end{scope}

    \node at ($(n)!0.5!(m)$) {$\Rightarrow$};
    \node at ($(m)!0.5!(mm)$) {$\Rightarrow$};

    \foreach \i [count=\xi from 1] in  {1,...,8}{
        \node[mycell, label=above:\footnotesize\xi] at (m-1-\i) {};
    }
    \foreach \i [count=\xi from 1] in  {1,...,8}{
        \node[mycell, label=above:\footnotesize\xi] at (n-1-\i) {};
    }
    \foreach \i [count=\xi from 1] in  {1,...,8}{
        \node[mycell, label=above:\footnotesize\xi] at (mm-1-\i) {};
    }
    \node[below=0mm of n] {$w=56723814$};
    \node[below=0mm of m] {$52163784$};
    \node[below=0mm of mm] {$w_e =12563478$};

  \end{tikzpicture}

\end{figure}

\subsection{Forest-like permutations and the Jordan-H\"older property}\label{binv:sec:A2}
Next, we will investigate elements $w \in W = S_{n+1}$ such that $\FF(w)$ satisfy the Jordan-H\"older property. By Theorem \ref{binv:thm:jhpmain}, this is equivalent to that $\#\supp(w) = \#\Binv(w)$.
By using \cite{sage}, we calculated the number of such elements in $S_{n+1}$ and obtained a sequence $2,6,22,89,379,1661,\dots$. This coincides with \cite[A111053]{OEIS}, a sequence of the number of \emph{forest-like permutations} defined in \cite{bmb}.
These conditions are equivalent (up to a multiplication by the longest element), as we shall see later.

The following is the result of \cite{bmb} in our context.
\begin{theorem}[{\cite[Theorem 1.1]{bmb}}]\label{binv:thm:bmb}
  Let $w$ be an element of $W = S_{n+1}$. Then the following are equivalent:
  \begin{enumerate}
    \item The Bruhat inversion graph $G_w$ is a forest, that is, it does not contain any cycles as an undirected graph.
    \item Define a map $L_w \colon \Z^n \to \Z^{\Binv(w)}$ by $L_w(e_i)=\sum \{\be \, | \,\be\in \Binv(w), i \in \supp(\be)\}$, where $e_i$ denotes the $i$-th standard basis of $\Z^n$. Then this map is surjective.
    \item $w$ avoids the patterns $4231$ and $3\ov{41}2$ with Bruhat restriction $4 \leftrightarrow 1$, that is, there exist no $1\leq a < b < c < d \leq n+1$ such that the one-line notation for $w$ is of the form $w = \cdots d \cdots b \cdots c \cdots a \cdots$ or $w= \cdots c \cdots d a \cdots b \cdots$.
  \end{enumerate}
\end{theorem}
We call $w$ \emph{forest-like} if $G_w$ is a forest. We remark that $w$ is forest-like in our sense if and only if $w_0 w$ is forest-like in the sense of \cite{bmb}, where $w_0 = (n+1)n \cdots 21$ is the longest element.

This turns out to be equivalent to our characterization of $w$ such that $\FF(w)$ satisfies (JHP):
\begin{proposition}\label{binv:prop:jhpforest}
  Let $w$ be an element of $W = S_{n+1}$. Then the following are equivalent:
  \begin{enumerate}
    \item $\FF(w)$ satisfies (JHP), that is, $\#\supp(w) = \#\BInv(w)$ holds.
    \item $w$ is forest-like.
  \end{enumerate}
\end{proposition}
\begin{proof}
  We will show that the surjectivity of $L_w \colon \Z^n \to \Z^{\Binv(w)}$ is equivalent to the bijectivity of the map $\varphi_w \colon \Z^{\Binv(w)} \to \Z^{\supp(w)}$ in Theorem \ref{binv:thm:jhpmain}.

  Since $\supp(w)$ is a subset of $Q_0 = \{1,2,\dots,n\}$, we have the natural inclusion $\iota \colon \Z^{\supp(w)} \hookrightarrow \Z^n$. Then it is clear from the definition of $L_w$ that $L_w$ is surjective if and only if so is $L'_w := L_w \circ \iota\colon \Z^{\supp(w)} \to \Z^{\Binv(w)}$.

  Now it is straightforward to check that $L'_w$ is nothing but the $\Z$-dual of $\varphi_w$ by calculating matrix representations of $L'_w$ and $\varphi_w$: the transpose of the matrix representing $L'_w$ coincides with the matrix representing $\varphi_w$.
  In particular, since $\varphi_w$ is always surjective, $L'_w$ should be injective. Therefore, the surjectivity of $L_w$ is equivalent to the bijectivity of $L'_w$, which is in turn equivalent to the bijectivity of $\varphi_w$, since $\Z$-dual $\Hom_\Z(-,\Z)$ is a duality of the category of finitely generated free $\Z$-modules.
\end{proof}
The motivation of forest-like permutations in \cite{bmb} comes from the study of Schubert varieties in the flag variety.
Consider the flag variety $\SL_{n+1}(\mathbb{C})/B$, where $B$ is the subgroup of upper triangular matrices.
For $w \in S_{n+1}$, let $e_w$ denote the permutation matrix for $w$. Then the \emph{Schubert variety} $X_w$ is the Zariski closure of the $B$-orbit of $e_w$ in the flag variety.

It is known that $X_w$ is a projective variety, which is not necessarily smooth. There are various studies on the relation between algebro-geometric properties of $X_w$ and combinatorial properties of $w$. For example, $X_w$ is smooth if and only if $w$ avoids the patterns $4231$ and $3412$.
Since there are many excellent papers and books on Schubert varieties, we only refer the reader to the recent survey article \cite{abebi} for the details.

We say that a variety is \emph{factorial} if the local ring at every point is a unique factorization domain. In \cite[Proposition 2]{wy}, it was proved that $X_w$ is factorial if and only if the map $L_w$ in Theorem \ref{binv:thm:bmb} is surjective. Thus Theorem \ref{binv:thm:bmb} characterizes the factoriality of $X_w$ in a combinatorial way.

By combining this result to Theorems \ref{binv:thm:jhpmain} and \ref{binv:thm:bmb}, Proposition \ref{binv:prop:jhpforest} and \cite[Theorem 3.1]{bmb}, we immediately obtain the following summary:
\begin{corollary}\label{binv:cor:typeAjhp}
  Let $w$ be an element of $W = S_{n+1}$. Then the following are equivalent:
  \begin{enumerate}
    \item $G_w$ is forest.
    \item $\#\supp(w) = \#\BInv(w)$ holds.
    \item Elements in $\Binv(w)$ are linearly independent.
    \item $\FF(w)$ satisfies (JHP).
    \item $X_w$ is factorial.
    \item $w$ avoids the pattern $4231$ and $3\ov{41}2$ with Bruhat restriction $4 \leftrightarrow 1$.
  \end{enumerate}
\end{corollary}
We remark that there is an explicit formula for the generation function of the number of forest-like permutations in \cite[(2)]{bmb}, and the number grows exponentially, which says that although there are lots of forest-like permutations, the number is relatively small compared to $\# S_{n+1} = (n+1)!$.

Theorem \ref{binv:thm:jhpmain} shows that (2)-(4) above are equivalent for all Dynkin types. It seems that they are also equivalent to (5) for other Dynkin types, by the same argument as in Proposition \ref{binv:prop:jhpforest} using the maps $\varphi_w$ in Theorem \ref{binv:thm:jhpmain} and $L_w$ in Theorem \ref{binv:thm:bmb} and the Monk-Chevalley formula.
Since the author does not have appropriate knowledge on algebraic geometries and Schubert varieties, we state it in the following conjecture:
\begin{conjecture}
  Let $G$ be a simple algebraic group over $\mathbb{C}$, $B$ a Borel subgroup, $\Phi$ the associated root system and $W$ the Weyl group. Fix $w \in W$, and denote by $X_w$ the Schubert variety in $G/B$. Then the following are equivalent:
  \begin{enumerate}
    \item $X_w$ is factorial.
    \item $\#\supp(w) = \#\Binv(w)$ holds, or equivalently, Bruhat inversions of $w$ are linearly independent, or equivalently, the torsion-free class $\FF(w)$ in $\mod \Pi_\Phi$ satisfies (JHP).
  \end{enumerate}
\end{conjecture}

For type D case, we can use the realization of the Weyl group as a group of signed permutations with even number of negatives. In particular, it seems to be an interesting problem to find an analogue of Corollary \ref{binv:cor:typeAjhp} for type D case.

\end{appendix}


\begin{thebibliography}{199}

  \bibitem[AB]{abebi}
  H. Abe, S. Billey,
  \emph{Consequences of the Lakshmibai-Sandhya theorem: the ubiquity of permutation patterns in Schubert calculus and related geometry},
  Adv. Stud. Pure Math., 71 (2016), Math. Soc. Japan, 1--52.

  \bibitem[AIR]{AIR}
  T. Adachi, O. Iyama, I. Reiten,
  \emph{$\tau$-tilting theory},
  Compos. Math. 150 (2014), no. 3, 415--452.

  \bibitem[AIRT]{AIRT}
  C. Amiot, O. Iyama, I. Reiten, G. Todorov,
  \emph{Preprojective algebras and $c$-sortable words},
  Proc. Lond. Math. Soc. (3) 104 (2012), no. 3, 513--539.

  \bibitem[Asa1]{asai}
  S. Asai, \emph{Semibricks}, Int. Math. Res. Not. rny150, 2018.

  \bibitem[Asa2]{asai2}
  S. Asai, \emph{Bricks over preprojective algebras and join-irreducible elements in Coxeter groups}, arXiv:1712.08311.

  \bibitem[ASS]{ASS}
  I. Assem, D. Simson, A. Skowro\'nski,
  \emph{Elements of the representation theory of associative algebras. Vol. 1. Techniques of representation theory},
  London Mathematical Society Student Texts, 65. Cambridge University Press, Cambridge, 2006. x+458 pp.

  \bibitem[BB]{bb}
  A. Bj\"orner, F. Brenti,
  \emph{Combinatorics of Coxeter groups},
  Graduate Texts in Mathematics, 231. Springer, New York, 2005. xiv+363 pp.

  \bibitem[BMB]{bmb}
  M. Bousquet-M\'elou, S. Butler,
  \emph{Forest-like permutations},
  Ann. Comb. 11 (2007), no. 3-4, 335--354.

  \bibitem[BHLR]{bhlr}
  T. Br\"ustle, S. Hassoun, D. Langford, S. Roy,
  \emph{Reduction of exact structures},
  J. Pure Appl. Algebra 224 (2020), no. 4, 106212, 29 pp.

  \bibitem[BST]{bst}
  T. Br\"ustle, D. Smith, H. Treffinger,
  \emph{Wall and chamber structure for finite-dimensional algebras},
  Adv. Math. 354 (2019), 106746, 31 pp.

  \bibitem[BIRS]{BIRS}
  A. B. Buan, O. Iyama, I. Reiten, J. Scott,
  \emph{Cluster structures for 2-Calabi-Yau categories and unipotent groups}, Compos. Math. 145 (2009), no. 4, 1035--1079.

  \bibitem[CB]{CB}
  W. Crawley-Boevey,
  \emph{On the exceptional fibres of Kleinian singularities},
  Amer. J. Math. 122 (2000), no. 5, 1027--1037.

  \bibitem[DIRRT]{DIRRT}
  L. Demonet, O. Iyama, N. Reading, I. Reiten, H. Thomas,
  Lattice theory of torsion classes,
  arXiv:1711.01785.

  \bibitem[DK]{DK}
  B. Keller, with an appendix by L. Demonet,
  \emph{A survey on maximal green sequences}, arXiv:1904.09247.

  \bibitem[DR]{dr}
  V. Dlab, C. M. Ringel,
  \emph{A module theoretical interpretation of properties of the root system},
  Ring theory (Proc. Antwerp Conf. (NATO Adv. Study Inst.), Univ. Antwerp, Antwerp, 1978), pp. 435--451, Lecture Notes in Pure and Appl. Math., 51, Dekker, New York, 1979.

  \bibitem[Dye]{dyer}
  M. J. Dyer,
  \emph{Hecke algebras and shellings of Bruhat intervals},
  Compositio Math. 89 (1993), no. 1, 91--115.


  \bibitem[Eno]{eno}
  H. Enomoto,
  \emph{The Jordan-H\"older property and Grothendieck monoids of exact categories}, arXiv:1908.05446.

  \bibitem[FS]{fs}
  C. K. Fan, J. R. Stembridge,
  \emph{Nilpotent orbits and commutative elements},
  J. Algebra 196 (1997), no. 2, 490--498.

  \bibitem[GL]{gl}
  R. M. Green, J. Losonczy,
  \emph{Freely braided elements of Coxeter groups},
  Ann. Comb. 6 (2002), no. 3-4, 337--348.

  \bibitem[GLS]{GLS}
  C. Gei\ss, B. Leclerc, J. Schr\"oer,
  \emph{Kac-Moody groups and cluster algebras},
  Adv. Math. 228 (2011), no. 1, 329--433.

  \bibitem[Hum1]{hum1}
  J. E. Humphreys,
  \emph{Introduction to Lie algebras and representation theory},
Second printing, revised. Graduate Texts in Mathematics, 9. Springer-Verlag, New York-Berlin, 1978. xii+171 pp.

  \bibitem[Hum2]{hum2}
  J. E. Humphreys,
  \emph{Reflection groups and Coxeter groups},
  Cambridge Studies in Advanced Mathematics, 29. Cambridge University Press, Cambridge, 1990. xii+204 pp.

  \bibitem[IT]{IT}
  C. Ingalls, H. Thomas,
  \emph{Noncrossing partitions and representations of quivers},
  Compos. Math. 145 (2009), no. 6, 1533--1562.

  \bibitem[IRRT]{IRRT}
  O. Iyama, N. Reading, I. Reiten, H. Thomas,
  \emph{Lattice structure of Weyl groups via representation theory of preprojective algebras},
  Compos. Math. 154 (2018), no. 6, 1269--1305.

  \bibitem[KS]{KS}
  M. Kashiwara, Y. Saito,
  \emph{Geometric construction of crystal bases},
  Duke Math. J. 89 (1997), no. 1, 9--36.

  \bibitem[Lus]{Lus}
  G. Lusztig,
  \emph{Canonical bases arising from quantized enveloping algebras II},
  Common trends in mathematics and quantum field theories (Kyoto, 1990). Progr. Theoret. Phys. Suppl. No. 102 (1990), 175--201 (1991).

  \bibitem[Miz]{mizuno}
  Y. Mizuno,
  \emph{Classifying $\tau$-tilting modules over preprojective algebras of Dynkin type},
  Math. Z. 277 (2014), no. 3-4, 665--690.

  \bibitem[OEIS]{OEIS}
   \emph{The On-Line Encyclopedia of Integer Sequences}, published electronically at https://oeis.org

  \bibitem[Pap]{papi}
  P. Papi,
  \emph{A characterization of a special ordering in a root system},
  Proc. Amer. Math. Soc. 120 (1994), no. 3, 661--665.

  \bibitem[Rea]{readingarc}
  N. Reading,
  \emph{Noncrossing arc diagrams and canonical join representations},
  SIAM J. Discrete Math. 29 (2015), no. 2, 736--750.

  \bibitem[Sage]{sage}
  The Sage Developers, \emph{SageMath, the Sage Mathematics Software System (Version 8.6)}, 2020, {\tt https://www.sagemath.org}.

  \bibitem[Tat]{tattar}
  A. Tattar,
  \emph{Torsion pairs and quasi-abelian categories},
  arXiv:1907.10025.

  \bibitem[Tho]{thomas}
  H. Thomas,
  \emph{Coxeter groups and quiver representations. Surveys in representation theory of algebras}, 173--186,
  Contemp. Math., 716, Amer. Math. Soc., Providence, RI, 2018.

  \bibitem[Tre]{tre}
  H. Treffinger,
  \emph{An algebraic approach to Harder-Narasimhan filtrations},
  arXiv:1810.06322.

  \bibitem[WY]{wy}
  A. Woo, A. Yong,
  \emph{When is a Schubert variety Gorenstein?},
  Adv. Math. 207 (2006), no. 1, 205--220.

\end{thebibliography}
\end{document}